\documentclass[11pt,reqno]{amsart}
\usepackage{microtype}

\usepackage[dvipsnames]{xcolor}
\usepackage
[colorlinks=true,linkcolor=Maroon,citecolor=OliveGreen,backref=page]
{hyperref}
\usepackage{amssymb}
\usepackage[shortlabels]{enumitem}
\setlist[enumerate]{label={(\arabic*)}}

\usepackage[capitalize]{cleveref}
\crefname{equation}{}{}

\usepackage[abbrev]{amsrefs}  % bibliography

\numberwithin{equation}{section}
\newtheorem{theorem}{Theorem}[section]
\newtheorem{lemma}[theorem]{Lemma}

\newtheorem{proposition}[theorem]{Proposition}
\newtheorem{corollary}[theorem]{Corollary}

\newtheorem{conjecture}{Conjecture}
\Crefname{conjecture}{Conjecture}{Conjectures}

\newtheorem*{theorem*}{Theorem}
\newtheorem*{question}{Question}

\theoremstyle{definition}

\theoremstyle{remark}
\newtheorem{remark}[theorem]{Remark}

\def\F{\mathbf{F}} % field
\def\T{\mathcal{T}} % tree
\def\B{\mathcal{B}} % block system

\def\G{\mathrm{G}} % spinal groups of type G

\newcommand\opr[1]{\operatorname{#1}}
\def\Sym{\opr{Sym}}
\def\Alt{\opr{Alt}}
\def\GL{\opr{GL}}
\def\PGL{\opr{PGL}}
\def\SL{\opr{SL}}

\def\Sp{\opr{Sp}}
\def\AGL{\opr{AGL}}
\def\Aut{\opr{Aut}}
\def\Out{\opr{Out}}

\def\diam{\opr{diam}}

\def\nsgp{\trianglelefteq}

\def\soc{\opr{soc}}
\def\sol{\opr{sol}}

\def\sublen{\lambda}

\def\nacomp{\mathcal{N}}

\def\mucf{\mu_{\opr{cf}}}
\def\muab{\mu_{\opr{ab}}}
\def\muna{\mu_{\opr{na}}}

\def\eps{\varepsilon}
\def\len{\ell}

\def\Gap{\opr{Gap}}

\newcommand\br[1]{{\left(#1\right)}}

\newcommand\gen[1]{\left\langle#1\right\rangle}
\newcommand\floor[1]{\left\lfloor#1\right\rfloor}
\newcommand\ceil[1]{\left\lceil#1\right\rceil}

\renewcommand{\subset}{\subseteq} % up for discussion

\usepackage{todonotes}

\numberwithin{equation}{section}

\author{Sean Eberhard}
\author{Elena Maini}
\author{Luca Sabatini}
\author{Gareth Tracey}

\address{\parbox{\linewidth}{Mathematics Institute, Zeeman Building, University of Warwick, UK \vspace{0.1cm}}}
\email{firstname.lastname@warwick.ac.uk}

\thanks{S.E.\ and L.S.\ are supported by the Royal Society through a University Research Fellowship held by S.E.\ (URF\textbackslash R1\textbackslash 221185). E.M.\ is supported by the Warwick Mathematics Institute Centre for Doctoral Training, and gratefully acknowledges funding by the Swinnerton-Dyer scholarship.}

\begin{document}
\title[Diameter bounds for arbitrary finite groups]{Diameter bounds for arbitrary finite groups and applications}

\begin{abstract}
    We prove a strong general-purpose bound for the diameter of a finite group depending only on the diameters of its composition factors and the maximal exponent of a normal abelian section.
    There are a number of notable applications:
    \begin{enumerate}
        \item if $G$ is a finite soluble group of exponent $e$, $\diam(G) \ll e (\log |G|)^8$,
        \item anabelian groups with bounded-rank composition factors have polylogarithmic diameter,
        \item transitive soluble subgroups of $S_n$ have diameter $\ll n^5$, and
        \item Grigorchuk's gap conjecture holds for any finitely generated group acting faithfully on a bounded-degree rooted tree.
    \end{enumerate}
    Additionally, conditional on Babai's conjecture,
    \begin{enumerate}[resume]
        \item any transitive permutation group of degree $n$ has diameter bounded by a polynomial in $n$ (a folkloric conjecture), and
        \item Grigorchuk's gap conjecture holds for residually finite groups, and thus the conjecture reduces to the simple case.
    \end{enumerate}
\end{abstract}

\maketitle
%\tableofcontents

\section{Introduction}
    If $G$ is a group and $X \subseteq G$ (usually a generating set), we write $\len_X$ for the length function with respect to $X$, i.e., if $g \in G$ then $\len_X(g)$ is the length of the minimal representation of $g$ as a product of elements of $X \cup X^{-1}$,
    or $\infty$ if none exists.
    For a subset $S \subseteq G$ we write $\len_X(S) = \sup_{g \in S} \len_X(g)$.
    The \emph{diameter} of a finite group $G$ with respect to $X$ is $\diam(G, X) = \len_X(G)$, and we define
    \[
        \diam(G) = \max \{\diam(G, X) : X~\text{a generating set for}~G\} .
    \]

    The most important conjectures about diameters of finite groups are the following two.
    The first is folkloric, but seems to have first taken definite form in the early 1980s: it is implicit in Driscoll--Furst~\cite{DF83}, and explicit in Kornhauser--Miller--Spirakis~\cite{KMS84} and McKenzie~\cite{McK84}.
    The second is universally attributed to Babai, even though it first appeared in a joint paper with Seress~\cite{BS92}.

    \begin{conjecture}[folklore]
        \label[conjecture]{conj:folk}
        If $G \le S_n$ is transitive then $\diam(G) \ll n^{O(1)}$.
    \end{conjecture}

    \begin{conjecture}[Babai]
        \label[conjecture]{conj:babai}
        If $G$ is a nonabelian finite simple group then $\diam(G) \ll (\log |G|)^{O(1)}$.
    \end{conjecture}

    Here and throughout the paper we are using standard big-$O$ notation, as well as the Vinogradov notation $A \ll B$, which means $A \le O(B)$.

    There are plenty of examples of generating sets for $A_n$ or $S_n$ giving quadratic diameter, and there are no known transitive permutation groups having a Cayley graph with greater than quadratic diameter.
    In both cases one may conjecture that the optimal exponent is $2$.

    In the direction of \Cref{conj:folk}, Helfgott--Seress~\cite{HS14} proved that for any transitive group $G \le S_n$,
    \begin{equation}
        \label{eq:helfgott--seress}
        \diam(G) \le n^{O((\log n)^3 \log \log n)}.
    \end{equation}
    This is the best known bound for $\diam(A_n)$, so \Cref{conj:babai} remains open in this case, as for the other simple groups of unbounded rank.
    Here and throughout the paper we use the term \emph{rank} for the degree $n$ in the case of $A_n$ or the (untwisted) Lie rank in the case of a simple group of Lie type.
    The rank of sporadic groups (or any finite subcollection) is unimportant and can be taken to be $1$ by convention.
    The simple groups of rank greater than $8$ are alternating or classical.
    The best known general bound for the diameter of a simple classical group $G \le \PGL(V)$ is
    \begin{equation}
        \label{eq:HMPQ}
        \diam(G) \le |V|^{O(\log \dim V)^2},
        %        \diam(G) < q^{O(n (\log n)^2)},
    \end{equation}
    which is due to Halasi--Mar\'oti--Pyber--Qiao~\cite{HMPQ}.

    \Cref{conj:babai} has been proved in bounded rank by Helfgott, Breuillard--Green--Tao, and Pyber--Szab\'o~\cites{BGT11,pyber-szabo-2016}.
    Thus for a simple group $G$ of rank $n$ we know that
    \[
        \diam(G) \le (\log |G|)^{c(n)}
    \]
    for some function $c(n)$ of the rank.
    Recently, it was proved by Bajpai--Dona--Helfgott~\cite{BDH} that
    \[
        \diam(G) \le (\log |G|)^{O(n^4)}
    \]
    for untwisted classical groups $G$ of rank $n$ defined over $\F_q$ of characteristic $p > n$, and there is considerable hope that this bound may be extended to the excluded cases.
    At present, in full generality, or for groups defined over a small field, the best available bound for a simple classical group $G \le \PGL(V)$ is still \eqref{eq:HMPQ}.

    The results in this paper are complementary to the case of simple groups.
    We give a general polylog-type bound for the diameter of a group in terms of the diameters of its composition factors as well as the maximal exponent of a normal abelian section.
    There are several notable applications.
    Among them is the striking fact that \Cref{conj:babai} implies \Cref{conj:folk}.

    To state our results it is useful to introduce some special notation.
    Let $G$ be a finite group.
    The standard notation $\mu(G)$ indicates the minimal degree of a faithful permutation representation $G \to S_n$.
    Note that if $G$ is simple then $\mu(G)$ is just the minimal index of a proper subgroup of $G$.
    Let $\nacomp(G)$ denote the set of nonabelian composition factors of $G$.
    We define
    \begin{align*}
        \theta_1(G) &= \max \left\{\frac{\log\diam(T)}{\log \mu(T)} : T \in \nacomp(G)\right\} \cup \{1\},\\
        \theta_2(G) &= \max \left\{\frac{\log\diam(T)}{\log \log |T|} : T \in \nacomp(G)\right\} \cup \{1\}.
    \end{align*}
    The idea is that $\theta_i(G)$ measures how much Conjecture $i$ fails for the composition factors of $G$.
    In particular, Conjecture $i$ implies that $\theta_i(G)$ is absolutely bounded.
    Moreover, by the results quoted above, $\theta_i(G)$ is in any case bounded by a function of the maximal rank of a nonabelian composition factor of $G$.
    Note that $\theta_1(G) \le 2 \theta_2(G)$, since $|T| \le \mu(T)!$.

    Recall that the \emph{exponent} $\exp(G)$ of a group $G$ is the smallest integer $e$ such that $g^e = 1$ for all $g \in G$.
    Let
    \[
        \eps(G) = \max \{\exp(N / N') : N \nsgp G\}.
    \]
    Here $N' = [N,N]$ is the derived subgroup of $N$.

    \begin{theorem}[Main Theorem]
        \label[theorem]{thm:main}
        If $G$ is a finite group, then
        \[
            \diam(G) \ll \eps(G) (\log |G|)^{31 + 5\theta_1(G)} \max\{\diam(T) : T \in \nacomp(G)\}\cup\{1\}.
        \]
    \end{theorem}

    We list a number of corollaries and subresults.
    To begin with, if $G$ is soluble then $\nacomp(G) = \emptyset$, so the theorem asserts that
    \[
        \diam(G) \ll \eps(G) (\log |G|)^{O(1)},
    \]
    i.e., soluble groups with controlled exponent have polylogarithmic diameter.
    In fact this is not just an important special case but also a crucial intermediate result, and we actually obtain the following explicit bounds.

    \begin{theorem}
        \label[theorem]{thm:main-soluble}
        Let $G$ be a finite soluble group. Let $L$ be the derived length and let
        \[
            \eps_0 = \max\{\exp(G^{(i)} / G^{(i+1)}) : 0 \le i < L\} \le \exp(G).
        \]
        Then
        \[
            \diam(G)
            \le \eps_0 4^{L-1} (\log_2|G|)^2
            \le \eps_0 (4 \log_2 |G|)^8.
        \]
    \end{theorem}

    Next, since $\diam(T) \le (\log |T|)^{\theta_2(G)}$ for $T \in \nacomp(G)$ and $\theta_1(G) \le 2 \theta_2(G)$, the following corollary is immediate from \Cref{thm:main}.

    \begin{corollary}
        \label[corollary]{cor:theta-2-cor}
        If $G$ is a finite group then
        \[
            \diam(G) \ll \eps(G) (\log |G|)^{31 + 11 \theta_2(G)}.
        \]
    \end{corollary}

    A finite group $G$ is called \emph{anabelian} if it has no abelian composition factors.

    \begin{corollary}
        \label[corollary]{cor:anabelian}
        If $G$ is anabelian then
        \[
            \diam(G) \ll (\log |G|)^{31 + 11 \theta_2(G)}.
        \]
    \end{corollary}

    Now let us turn to the case of transitive permutation groups and \Cref{conj:folk}, which is our primary motivation.

    \begin{corollary}[\Cref{conj:babai} $\implies$ \Cref{conj:folk}]
        \label[corollary]{cor:transitive}
        If $G \le S_n$ is transitive then
        \[\diam(G) \ll n^{32 + 7 \theta_1(G)}.\]
        In particular,
        \Cref{conj:folk} holds if and only if the following two special cases hold:
        \begin{enumerate}[(a)]
            \item $\diam(A_n) \le n^{O(1)}$,
            \item $\diam(G) \le |V|^{O(1)}$ for simple classical groups $G \le \PGL(V)$.
        \end{enumerate}
    \end{corollary}

    The best known results in the direction of (a) and (b) are \eqref{eq:helfgott--seress} and \eqref{eq:HMPQ} respectively.
    In view of the latter, we are tantalizingly close to reducing \Cref{conj:folk} to the case of $A_n$.
    For now the best we can say is the following.

    \begin{corollary}
        \label[corollary]{cor:transitive-2}
        Let $G \le S_n$ be a transitive permutation group. Let $m \ge 5$ be the largest integer such that $A_m \in \nacomp(G)$, or else let $m = 1$ if $G$ has no alternating composition factors.
        Then
        \[
            \diam(G) \ll n^{O((\log \log n)^2 + (\log m)^3 \log \log m)}.
        \]
    \end{corollary}

    \Cref{cor:transitive} implies in particular that if the nonabelian composition factors of $G \le S_n$ have bounded rank then $\diam(G) \ll n^{O(1)}$. In particular if $G$ is soluble then $\diam(G)$ is polynomially bounded. A direct application of \Cref{cor:transitive} gives $\diam(G) \ll n^{39}$, but in fact we can obtain $\diam(G) \ll n^5$ by being more careful with our tools.

    \begin{theorem}
        \label[theorem]{cor:soluble-transitive}
        If $G \le S_n$ is soluble and transitive then $\diam(G) \ll n^5$.
    \end{theorem}

    Let us briefly give the idea of the proof of \Cref{thm:main} as well as its corollaries.
    The starting point is that there is a subadditive extension bound for lengths: for any normal subgroup $N \nsgp G$,
    \[
        \len_X(G) \le \len_X(N) + \len_X(G/N).
    \]
    Here $\len_X(G/N)$ is just the diameter of the quotient $G/N$ with respect to the projection $XN/N$ of the generators, so
    \[
        \len_X(G/N) \le \diam(G/N).
    \]
    Bounding the term $\len_X(N)$ is tricky, however, as $X$ is a generating set for $G$, not $N$, but at least
    \[
        \len_X(N) \le m \diam(N)
    \]
    provided that $N$ has a generating set $Y$ with $\len_X(Y) \le m$.
    In general a basic lemma of Schreier gives $m \le 2\diam(G/N) + 1$, and this implies
    \[
        \diam(G) \ll \diam(N) \diam(G/N).
    \]
    However, because the diameters are multiplied in this bound,
    this is an expensive inequality that we must apply sparingly.
    Our arguments therefore focus on constructing generating sets for small normal subgroups by more efficient means than Schreier's lemma.

    In the case of a soluble group, we achieve this by descending the derived series using commutators and conjugation in a natural way,
    taking advantage of an old lemma of Milnor
    (used in the proof of the Milnor--Wolf theorem).
    We also take advantage of the fact that the derived length cannot be very large: the derived length of a soluble group $G$ is bounded by $O(\log \log |G|)$ by a result of Glasby~\cite{glasby}.
    If $G$ happens to be a soluble transitive permutation group of degree $n$ then we have even sharper bounds for its derived length, and also for the exponent parameter $\eps(G)$.

    In the case of an arbitrary finite group, we first consider the quotient of $G$ by its soluble radical.
    The quotient embeds as a subdirect product of a finite collection of monolithic groups with nonabelian socles.
    The prototypical example of such a group is the iterated wreath product
    \[
        G = A_5 \wr \cdots \wr A_5.
    \]
    For such groups we use an inductive argument to bound the length of some nontrivial element of the socle.
    For this we rely on a generalization of an argument of Wilson from \cite{wilson2}
    (in fact, the study of Wilson's papers~\cites{wilson1,wilson2} motivated much of the present paper).

    There is a well-known theorem of P\'alfy--Wolf~\cites{palfy,wolf} that asserts that the order of a primitive soluble group $G \le S_n$ is bounded by a polynomial function of $n$; in fact $|G| \le 24^{-1/3} n^c$, where $c \approx 3.25$.
    This influential result was extended by Pyber~\cite{Py93}*{Theorem~2.10} (see also Guralnick--Mar\'oti--Pyber~\cite{Gur-Mar-Pyb}), who showed that if $G \le S_n$ is primitive then the product of the orders of its abelian composition factors is bounded by the same function,
    and moreover that the number of nonabelian composition factors is $O(\log n)$.
    In the analysis of monolithic groups with nonabelian socles, it turns out, somewhat surprisingly, that we need the following natural generalization of these results.

    \begin{theorem}[Generalized P\'alfy--Wolf Theorem]
        \label[theorem]{thm:gen-palfy-wolf}
        If $G \le S_n$ is a primitive permutation group of degree $n$ with composition series $G = G_0 > \cdots > G_\ell = 1$ and factors $T_i = G_{i-1} / G_i$ then
        \(
            \prod_{i=1}^\ell \mu(T_i) < n^5.
        \)
    \end{theorem}

    This result may be of independent interest. For example, it implies the theorem of Babai--Cameron--P\'alfy~\cite{BCP82} with an exponent linear in the rank, since $\mu(T) \ge |T|^{c/r}$ for simple groups of rank at most $r$. (This also follows from Liebeck--Shalev~\cite{liebeck-shalev-1999}*{Theorem~1.4}.)

    \begin{corollary}[Babai--Cameron--P\'alfy with linear exponent]
        \label[corollary]{cor:explicit-BCP}
        There is a constant $C$ such that the following holds.
        If $G \le S_n$ is a primitive permutation group of degree $n$ and $G$ has no nonabelian composition factor of rank greater than $r$ then $|G| \le n^{Cr}$.
    \end{corollary}

    \Cref{thm:gen-palfy-wolf} and an iterative application of Schreier's lemma give a bound on the diameter of a primitive permutation group in terms of $\theta_1(G)$. This special case of \Cref{cor:transitive} is used in the proof of \Cref{thm:main}.
    To deal with an arbitrary finite group, we ascend a normal series and alternately cover an insoluble normal subgroup or apply \Cref{thm:main-soluble}.

    We now discuss important applications of our results
    to the growth of residually finite groups.
    If $G$ is a group generated by a finite set $X$, the \emph{growth function of $X$} (or \emph{of $G$}, by abuse) is defined by
    \[
        \gamma_X(n) = |B_X(n)|, \qquad B_X(n) = \{ g\in G:\len_X(g)\le n \}.
    \]
    Growth functions are conventionally compared in the following way: we write $f \preceq g$ if there is a constant $C > 0$ such that $f(n) \le C g(Cn)$ for all $n > 0$.
    This relation defines a preorder on the set of growth functions, and the equivalence class $[\gamma_X]$ is an invariant of $G$, i.e., independent of the choice of generating set $X$; in fact it is a quasi-isometry invariant.
    Growth functions of typical examples (abelian groups, soluble groups, linear groups, ...) are generally either polynomially bounded or exponential.
    A group $G = \gen X$ is said to have \emph{intermediate growth} in all other cases, i.e., if $n^d \prec \gamma_X(n) \prec \exp(n)$ for all $d$.

    It is not difficult to prove that a virtually nilpotent group has polynomial growth, i.e., $\gamma_X(n) \preceq n^d$ for some integer $d$.
    Gromov~\cite{gromov} famously proved the converse: if $G$ has polynomial growth then $G$ is virtually nilpotent.
    \emph{Grigorchuk's gap conjecture} predicts the existence of a constant $\beta > 0$ such that the same conclusion holds under the much weaker hypothesis that $\gamma_X(n) \prec \exp(n^\beta)$.
    The strong version asserts that one can take $\beta = 1/2$.
    See the introduction to \cite{EM25} for more on the history of this question,
    and see \cite{EM25}*{Section~5} for even stronger variants.

    \begin{conjecture}[Grigorchuk]
        \label[conjecture]{conj:gap}
        There is a constant $\beta > 0$ such that the following holds.
        If $G$ is a group generated by a finite set $X$ and $\gamma_X(n) \prec \exp(n^\beta)$ then $G$ is virtually nilpotent.
    \end{conjecture}

    As observed by Grigorchuk (see \cite{grig-ICM}*{Theorem~7.4}), the main result of Bajorska--Macedo\'nska~\cite{bajorska--macedonska} implies that the gap conjecture reduces to two extreme cases: residually finite groups and simple groups.
    For further reduction theorems in this spirit, see \cite{grig-gaps-and-JI}.
    All the early examples of groups of intermediate growth were residually finite.
    Non-residually-finite examples were first constructed by Erschler~\cite{erschler-2004} and Bartholdi--Erschler~\cite{bartholdi-erschler-2012}.
    The first simple group of intermediate growth was constructed by Nekrashevych~\cite{nekrashevych} (solving a long-standing open problem).

    Using our methods we can show that Babai's conjecture (\cref{conj:babai}) implies Grigorchuk's gap conjecture for residually finite groups.
    We formulate the result in an unconditional way as follows.

    \begin{theorem}[\Cref{conj:babai} $\implies$ \Cref{conj:gap} for residually finite groups]
        \label[theorem]{thm:gap-RF}
        For any $\theta > 0$ there is some $\beta = 1/(20 + 10 \theta) > 0$ such that the following holds.
        If $G = \gen X$ is a finitely generated group such that
        \begin{enumerate}
            \item $G$ has a sequence of finite-index normal subgroups $N_i$ such that $\bigcap_{i=1}^\infty N_i = 1$ and such that $\theta_2(G / N_i) \le \theta$ for all $i \ge 1$, and
            \item $\gamma_X(n) \preceq \exp(n^\beta)$,
        \end{enumerate}
        then $G$ is virtually nilpotent.
    \end{theorem}

    In particular, we have the following notable corollary, which applies to many natural branch groups, for example all spinal groups with finite rooted and directed parts as defined in \cite{BGS-branch-groups-book}*{Chapter~2}.
    By \cite{BGS-branch-groups-book}*{Theorem~10.9}, this includes a rich class of groups of intermediate growth.

    \begin{corollary}
        \label[corollary]{cor:branch-groups}
        For $d \ge 1$ there is a constant $\beta = \beta(d) > 0$ such that the following holds.
        If $G = \gen X$ is a finitely generated group such that
        \begin{enumerate}
            \item $G$ acts faithfully on an infinite rooted $d$-regular tree, and
            \item $\gamma_X(n) \preceq \exp(n^\beta)$,
        \end{enumerate}
        then $G$ is virtually nilpotent.
        In fact we can take $\beta(d) = 1 / (\log d)^4$ for $d$ sufficiently large.
    \end{corollary}

\section{Basic diametry}

The following easy lemmas are our most basic tools for bounding diameters. We will use them throughout the paper usually without explicit reference.

\begin{lemma}[chain rule]
    \label[lemma]{lem:chain-rule}
    For $X, Y, Z \subseteq G$ we have
    \[
        \len_X(Z) \le \len_X(Y) \len_Y(Z).
    \]
\end{lemma}

Let $B_X(r) = (X \cup X^{-1} \cup \{1\})^r$ denote the ball of radius $r$.
Generalizing the $\len_X$ notation, for any pair of subgroups $K \le H \le G$, let $\len_X(H/K)$ denote the smallest integer $r \ge 0$ such that $H \subseteq B_X(r) K$.

\begin{lemma}[extension rule]
    \label[lemma]{lem:extension-rule}
    Suppose $K \le L \le H \le G$. Then
    \[
        \len_X(H/K) \le \len_X(H/L) + \len_X(L/K).
    \]
\end{lemma}
\begin{proof}
    If $r_1 = \len_X(H/L)$ and $r_2 = \len_X(L/K)$ then
    \[
        H \subseteq B_X(r_1) L \subseteq B_X(r_1) B_X(r_2) K = B_X(r_1 + r_2) K.\qedhere
    \]
\end{proof}

We often apply \Cref{lem:extension-rule} in the following iterated form.
Suppose $G = G_0 > G_1 > \cdots > G_L = 1$ and $X \subseteq G$. Then
\[
    \len_X(G) \le \sum_{i=1}^L \len_X(G_{i-1} / G_i).
\]

\begin{lemma}[Schreier]
    \label[lemma]{lem:schreier}
    Let $G = \gen X$ be a finite group and $N \nsgp G$.
    \begin{enumerate}
        \item $N = \gen Y$ where $\len_X(Y) \le 2 \diam(G/N) + 1$.
        \item $\diam(G) \le 4 \diam(N) \diam(G/N)$ provided $1 < N < G$.
    \end{enumerate}
\end{lemma}
\begin{proof}
    (Cf.~\cite{BS92}*{Lemma~5.1} or \cite{wilson2}*{Lemma~2.1}.)
    Schreier's classical lemma states that if $T$ is a transversal to the right cosets of $N$ then
    \[Y = \{t_1 x t_2^{-1} : x \in X, t_1, t_2 \in T\} \cap N\]
    generates $N$. A transversal of minimal $X$-length is $T = B_X(\diam(G/N))$. For this choice, $\len_X(Y) \le 2 \diam(G/N) + 1$, which proves \emph{(1)}.
    Part \emph{(2)} follows from \emph{(1)} and the previous two lemmas:
    \begin{align*}
        \len_X(G)
        &\le \len_X(G/N) + \len_X(Y) \len_Y(N) \\
        &\le \diam(G/N) + (2 \diam(G/N) + 1) \diam(N)\\
        &\le 4 \diam(G/N) \diam(N).\qedhere
    \end{align*}
\end{proof}

\begin{remark}
    The previous lemma is valid without the normality hypothesis if $\diam(G/N)$ is understood as the diameter of the Schreier graph.
\end{remark}

\section{Soluble groups}
\label[section]{sec:soluble}

The key to the proof of \Cref{thm:main-soluble} is the following short argument.
We refer to it as the \emph{generalized Milnor lemma}, as it is similar to \cite{milnor}*{Lemma~1}.
A closely related lemma was also critical in the recent paper \cite{EM25} of the first two authors.

If $G$ is a finite group, the \emph{subgroup length} of $G$ is the largest integer $\sublen(G) = \lambda$ such that there is a chain of subgroups $1 = H_0 < H_1 < \cdots < H_\lambda = G$.
We will use only the trivial bound $\lambda(G) \le \log_2 |G|$.

\begin{lemma}[Generalized Milnor lemma]
    \label[lemma]{lem:generalized-Milnor}
    Let $G = \gen X$ be a group and let $N = \gen {Y^G} \nsgp G$ be a finite normal subgroup.
    Then $N = \gen Z$ where
    \[
        \len_X(Z) \le \len_X(Y) + 2 \sublen(N).
    \]
\end{lemma}
\begin{proof}
    Let
    \[
        Z_i = \{y^g : y \in Y, g \in B_X(i)\}, \qquad N_i = \gen{Z_i}.
    \]
    Then
    \[
        \gen Y = N_0 \le N_1 \le N_2 \le \cdots \le N = \gen{Y^G},
    \]
    so there is some $k \le \sublen(N)$ such that $N_k = N_{k+1}$.
    Then $N_k \nsgp G$, so $N_k = N$.
%    If $g \in X \cup X^{-1}$ then
%    \[
%        Z_k^g \subseteq Z_{k+1} \subseteq N_{k+1} = N_k.
%    \]
%    Thus $N_k^g \le N_k$, and since $X$ generates $G$ it follows that $N_k \nsgp G$, so $N_k = N$.
    Thus $N = \gen Z$ where $Z = Z_k$ and
    \[
        \len_X(Z) \le \len_X(Y) + 2k \le \len_X(Y) + 2 \sublen(N).\qedhere
    \]
\end{proof}

If a group $N$ is generated by a subset $Y$,
in general its derived subgroup $N'$ is not generated
by commutators of elements of $Y$.
Hence the following result has nontrivial content.

\begin{lemma}
    \label[lemma]{lem:gen-Milnor-N'}
    Let $G = \gen X$ and let $N = \gen Y \nsgp G$ be a normal subgroup with $N'$ finite.
    Then $N' = \gen Z$ where
    \[
        \len_X(Z) \le 4\len_X(Y) + 2 \sublen(N').
    \]
\end{lemma}

\begin{proof}
    Let $Z_0 = \{[y_1, y_2] : y_1, y_2 \in Y\}$. Then $\len_X(Z_0) \le 4 \len_X(Y)$ and $Z_0$ normally generates $N'$.
    Apply \Cref{lem:generalized-Milnor} to $N' = \gen{Z_0^G}$.
\end{proof}

The derived series of $G$ is defined as usual by $G^{(0)} = G$ and $G^{(i)} = (G^{(i-1)})'$ for $i \ge 1$. Recall that $G$ is soluble if and only if $G^{(i)} = 1$ for some index $i$. The minimal such $i$ is called the \emph{derived length} of $G$.

\begin{lemma}
    \label[lemma]{lem:Xi}
    Let $G = \gen X$ be a finite group. For each $i \ge 0$, there is a set $X_i \subseteq G^{(i)}$ such that
    \[
        G^{(i)} = \gen{X_i}, \qquad \len_X(X_i) \le 4^i \log_2 |G|.
    \]
\end{lemma}
\begin{proof}
    Applying \Cref{lem:gen-Milnor-N'} inductively starting with $X_0 = X$ produces a generating set $X_i$ for $G^{(i)}$ such that
    \[
        \len_X(X_i) \le 4 \len_X(X_{i-1}) + 2 \sublen(G').
    \]
    Rewriting this inequality in the telescopic form
    \[
        \frac{\len_X(X_i)}{4^i} - \frac{\len_X(X_{i-1})}{4^{i-1}} \le \frac 2{4^i} \sublen(G'),
    \]
    we deduce that
    \[
        \len_X(X_i) / 4^i \le 1 + \frac23 \sublen(G')
        \le 1 + \frac23 \log_2 |G'|
        \le \log_2|G|
    \]
    for $G \ne 1$.
\end{proof}

\begin{lemma}
    \label[lemma]{lem:diam-abelian}
    If $A$ is a finite abelian group then
    \[
        \diam(A) \le \exp(A) \sublen(A) \le \exp(A) \log_2 |A|.
    \]
\end{lemma}
\begin{proof}
    Let $X$ be a generating set for $A$. We may assume $X$ is minimal. Then $|X| \le \sublen(A)$ and $A$ is the set of products $\prod_{x \in X} x^{e_x}$ with $0 \le e_x < \exp(A)$. Every such product has length at most $\exp(A) |X| \le \exp(A) \sublen(A)$.
\end{proof}

\begin{remark}
    In fact if $A \cong C_{d_1} \times \cdots \times C_{d_n}$ where $d_1 \mid \cdots \mid d_n$ then
    \[
        \diam(A) = \sum_{i=1}^n \floor{d_i / 2}.
    \]
    See \cite{wilson1}*{Section~3}.
\end{remark}

We can now prove our main theorem for soluble groups.

\begin{proof}[Proof of \Cref{thm:main-soluble}]
    Let $G = \gen X$ be a finite soluble group of derived length $L$.
    By a result of Glasby~\cite{glasby},
    \[L < 3\log_2\log_2|G| + 9.\]
    By \Cref{lem:Xi}, $G^{(i)} = \gen {X_i}$ where $\len_X(X_i) \le 4^i \log_2|G|$. Applying \Cref{lem:chain-rule,lem:extension-rule,lem:diam-abelian},
    \begin{align*}
        \len_X(G)
        &\le \sum_{i=0}^{L-1} \len_X(G^{(i)} / G^{(i+1)})
        \le \sum_{i=0}^{L-1} \len_X(X_i) \diam(G^{(i)} / G^{(i+1)}) \\
        &\le \sum_{i=0}^{L-1} 4^i \log_2|G| \exp(G^{(i)} / G^{(i+1)}) \log_2|G^{(i)} : G^{(i+1)}| \\
        &\le \eps_0 4^{L-1}  \log_2 |G| \sum_{i=0}^{L-1} \log_2|G^{(i)} : G^{(i+1)}| \\
        &= \eps_0 4^{L-1} (\log_2|G|)^2
        \le \eps_0 (4 \log_2|G|)^8.\qedhere
    \end{align*}
\end{proof}

\section{Generalized P\'alfy--Wolf}
\label[section]{sec:palfy-wolf}

If $T$ is a finite simple group (possibly abelian), we denote by $\mu(T)$ the minimal degree of a nontrivial permutation representation of $T$. Equivalently, \[\mu(T) = \min\{|T:H| : H < T\}.\]
Now if $G$ is an arbitrary finite group with composition series
\[
    G = G_0 > G_1 > \cdots > G_\ell = 1,
\]
with quotients $T_i = G_{i-1} / G_i$ for $1 \le i \le \ell$ (the composition factors, with multiplicity),
we define
\[
    \mucf(G) = \prod_{i=1}^\ell \mu(T_i).
\]
Note that $\mucf(G) = |G|$ if (and only if) $G$ is soluble.
Let us also define $\mathcal{A} = \{i : T_i~\text{is abelian}\}$ and
\[
    \muab(G) = \prod_{i \in \mathcal{A}} \mu(T_i) = \prod_{i \in \mathcal{A}} |T_i|,
    \qquad
    \muna(G) = \prod_{i \notin \mathcal{A}} \mu(T_i).
\]
Then
\begin{equation}
    \label{eq:muab-muna}
    \mucf(G) = \muab(G) \muna(G),
\end{equation}
and, by the theorem of Pyber quoted in the introduction, if $G \le S_n$ is primitive then
\begin{equation}
    \label{eq:pyber}
    \muab(G) \le 24^{-1/3} n^c,
\end{equation}
where \(c = 1 + c_0 = 1 + \log_9(48 \cdot 24^{1/3}) \approx 3.44\)
as in P\'alfy--Wolf.

The following two basic lemmas hold for each of $\mucf, \muab, \muna$, but we only need the versions for $\muna$.

\begin{lemma}
    \label[lemma]{lem:mucf-ext}
    If $G$ is a finite group and $N\nsgp G$ then
    \[\muna(G) = \muna(N) \muna(G/N).\]
\end{lemma}

\begin{lemma}
    \label[lemma]{lem:mucf-subdirect}
    If $G$ is a subdirect subgroup of $G_1 \times \dots \times G_t$, then
    \[
        \muna(G) \le \prod_{i=1}^t \muna(G_i).
    \]
\end{lemma}

\begin{proof}
    We may assume $t > 1$.
    Let $N = G \cap G_1$.
    Then $N \nsgp G$ and $G / N$ is subdirect in $G_2 \times \cdots \times G_t$, so by \Cref{lem:mucf-ext} and induction we have
    \[
        \muna(G) = \muna(N) \muna(G/N) \le \muna(N) \muna(G_2) \cdots \muna(G_t).
    \]
    Also, $N \nsgp G_1$ since $G$ is subdirect, so $\muna(N) \le \muna(G_1)$.
\end{proof}

We need the following standard result on the structure of primitive linear groups.
The main ideas are due to Suprunenko~\cites{sup1,sup2}, though Suprunenko only considered soluble groups;
more recent work tends to refer to Aschbacher~\cite{aschbacher} for much the same information.
The statement below can be found for example in \cite{Tracey}*{Section~2.3} with suitable references.

\begin{theorem}[Suprunenko]
    \label[theorem]{thm:suprunenko}
    Let $G \le \GL_n(p)$ be a primitive linear group, where $n\ge 1$ and $p$ is a prime.
    Then there is a divisor $r \mid n$ and a subgroup $H\nsgp G$ such that $G/H$ is cyclic of order $n/r$,
    $Z(H)$ is cyclic of order dividing $p^r-1$,
    and the generalized Fitting subgroup $F = F^*(H)$ is a central product
    \[
        F = Z(H) \circ E_1 \circ \cdots \circ E_m \circ U_1 \circ \cdots \circ U_s,
    \]
    where
    \begin{enumerate}[(i)]
        \item for $1 \le i \le m$, $E_i\nsgp H$ is an extraspecial group of order $p_i^{1+2n_i}$ for some prime divisor $p_i$ of $p^r - 1$ and some $n_i \ge 1$,
        \[
            p_i^{2n_i} \le H / C_H(E_i) \le p_i^{2n_i} \cdot \Sp_{2n_i}(p_i),
        \]
        and the projection of $H / C_H(E_i)$ to $\Sp_{2n_i}(p_i)$ is completely reducible;
        \item for $1 \le j \le s$, $U_j\nsgp H$ is a central product of some $u_j\ge 1$ copies of a quasisimple group $Q_j = Z(Q_j) \cdot T_j$,
        \[
            T_j^{u_j}\le H/C_H(U_j)\le \Aut(T_j)\wr \Sym(u_j),
        \]
        the projection of $H/C_H(U_j)$ to $\Sym(u_j)$ is transitive,
        and $Q_j$ has an irreducible representation $Q_j \to \SL_{t_j}(p)$ for some $t_j \ge 2$;
        \item the product $p_1^{n_1} \cdots p_m^{n_m} \cdot t_1^{u_1} \cdots t_s^{u_s}$ is a divisor of $n/r$.
    \end{enumerate}
\end{theorem}

We are going to bound $\muna(G)$ for primitive groups.
We actually prove a more general statement which is suitable to an inductive argument.

\begin{theorem}
    \label[theorem]{thm:muna-bounds}
    Let $b_1 = 5^{1/4} \approx 1.495$ and $c_1 = \log_8(7b_1) \approx 1.129$. Let $n \ge 1$.
    \begin{enumerate}
        \item If $G \le S_n$ is primitive (with $n \ge 2$) then $\muna(G) \le b_1^{-1} n^{5/4}$.
        \item If $G \le S_n$ then $\muna(G) \le b_1^{n-1}$.
        \item If $G \le \GL_n(p)$ is completely reducible then $\muna(G) \le b_1^{-1} p^{c_1 n}$.
    \end{enumerate}
\end{theorem}

\begin{remark}
    All three bounds are sharp for infinitely many $n$. Let
    \[
        G_h = A_5 \wr \cdots \wr A_5
    \]
    be the iterated wreath product of $h$ copies of $A_5$ (with the natural action).
    Then
    \[
        \muna(G_h) = 5^{5^{h-1} + \cdots + 5 + 1} = 5^{(5^h - 1) / 4}.
    \]
    There is a natural faithful transitive action of $G_h$ on $n = 5^h$ points (imprimitive for $h > 1$), which shows that the bound in \emph{(2)} is sharp.
    There is also a faithful primitive action of $G_h = A_5 \wr G_{h-1}$ on $n = 5^{5^{h-1}}$ points,
    which shows that the bound in \emph{(1)} is sharp.
    Finally, the group $\GL_3(2) \wr G_h$ acts faithfully and irreducibly on $\F_2^n$ where $n = 3 \cdot 5^h$, and
    \[
        \muna(\SL_3(2) \wr G_h) = 7^{5^h} 5^{(5^h-1) / 4} = b_1^{-1} 2^{c_1 n},
    \]
    which shows that the bound in \emph{(3)} is sharp.
\end{remark}

\begin{proof}
    First let us describe the structure of the induction.
    In the proof of each statement we assume that \emph{(1)}--\emph{(3)} hold for all positive integers $m < n$.
    Additionally, in the proof of \emph{(2)} we assume that \emph{(1)} holds for $m = n$.
    Finally, in the proof of \emph{(3)} we assume that \emph{(1)} and \emph{(2)} hold for $m = n$,
    and also that \emph{(3)} holds for $m = n$ and all smaller primes $p' < p$.

    \emph{(1)} Let $G \le \Sym(\Omega)$ be primitive, where $|\Omega| = n$.

    \emph{Affine case.} If $G$ has affine type then $G \cong VH$ where $V = \F_p^d$ may be identified with $\Omega$ and $H \le \GL(V)$ is irreducible. Since $d < p^d = n$, by induction we have
    \[
        \muna(G) = \muna(H) \le b_1^{-1} p^{c_1 d} = b_1^{-1} n^{c_1} < b_1^{-1}n^{5/4}.
    \]

    \emph{Non-affine case.}
    Now assume that $\soc(G)$ is nonabelian, say $\soc(G) = T^d$ where $T$ is a nonabelian simple group.
    Then
    \[T^d \le G \le \Aut(T^d) \cong \Aut(T) \wr S_d.\]
    Since $G$ is primitive, $T^d$ is transitive on $\Omega$.
    By \cite{kleidman--liebeck}*{5.2.7(ii)} (and the fact that $\mu(T) \le |T|^{1/2}$ \cite{lev92}),
    it follows that
    \[5^d \le \mu(T)^d \le n.\]
    In particular $d < n$.
    Let $\bar G$ be the image of $G$ in $S_d$.
    Now $\Out(T) = \Aut(T) / T$ is soluble by the Schreier conjecture, so by induction
    \[
        \muna(G) = \mu(T)^d \muna(\bar G)
        \le n b_1^{d-1}
        \le b_1^{-1} n^{5/4}
    \]
    (with equality if $n = 5$).

    \emph{(2)} Let $G \le \Sym(\Omega)$ where $|\Omega| = n$. If $G$ is intransitive then $\Omega = \Omega_1 \cup \Omega_2$ where $\Omega_1, \Omega_2$ are nonempty $G$-invariant subsets.
    Let $n_i = |\Omega_i|$ and let $G_i$ be the group induced on $\Omega_i$ for $i = 1,2$. Then $G$ is subdirect in $G_1 \times G_2$, so we are done by \Cref{lem:mucf-subdirect} and induction:
    \[
        \muna(G) \le \muna(G_1) \muna(G_2) \le b_1^{n_1 - 1} b_1^{n_2-1} < b_1^{n-1}.
    \]

    Similarly, if $G$ is imprimitive, consider an imprimitivity system $\B = \{B_1, \dots, B_d\}$ for $G$, where $d \mid n$ and $1 < d < n$. Then $G$ acts on $\B$ with kernel $N \nsgp G$ say.
    Let $N_i$ be the group induced on $B_i$ by $N$ for $1 \le i \le d$.
    Then $N_i \cong N_1$ for each $i$ and $N$ is subdirect in $N_1 \times \cdots \times N_d$, so by \Cref{lem:mucf-ext,lem:mucf-subdirect} and induction we have
    \[
        \muna(G) \le \muna(N_1)^d \muna(G/N) \le (b_1^{n/d-1})^d b_1^{d-1} = b_1^{n-1}.
    \]

    Hence we may assume that $G$ is primitive. We may assume $n \ge 5$.
    Then by \emph{(1)} we have
    \[
        \muna(G) \le b_1^{-1} n^{5/4} \le b_1^{n-1}
    \]
    (again equality holds if $n = 5$).

    \emph{(3)} Now let $G \le \GL(V)$ be a completely reducible linear group, where $V = \F_p^n$.
    If $G$ is reducible then $V = U_1 \oplus U_2$ for some $G$-submodules $U_1, U_2 \le V$,
    and $G$ is subdirect in $G_1 \times G_2$ where $G_i = G |_{U_i}$ is the group induced by $G$ on $U_i$.
    By \Cref{lem:mucf-subdirect} and induction it follows that
    \[
        \muna(G) \le \muna(G_1) \muna(G_2) \le b_1^{-2} |U_1|^{c_1} |U_2|^{c_1} < b_1^{-1} |V|^{c_1}.
    \]

    Next suppose that $G$ is irreducible but imprimitive on $V$.
    Then there is a $G$-invariant decomposition $V = U_1 \oplus \cdots \oplus U_d$ such that $1 < d \mid n$, $\dim U_i = n / d$ for each $i$, and $G$ permutes $U_1, \dots, U_d$ transitively with kernel $N \nsgp G$ say.
    By Clifford's theorem, $V$ is a completely reducible $N$-module.
    Let $N_i = N|_{U_i}$ be the completely reducible linear group induced by $N$ on $U_i$.
    Then $N_i \cong N_1$ for each $i$ and $N$ is subdirect in $N_1 \times \cdots \times N_d$.
    Hence, by \Cref{lem:mucf-ext,lem:mucf-subdirect} and induction,
    \[
        \muna(G)
        = \muna(N) \muna(G/N)
        \le \br{b_1^{-1} |U_1|^{c_1}}^{d} b_1^{d-1}
        = b_1^{-1} |V|^{c_1}.
    \]

    Finally, assume that $G$ is primitive on $V$.
    By \Cref{thm:suprunenko}, there is a divisor $r \mid n$ and a subgroup $H \nsgp G$ such that $G/H$ is abelian and
    \[
        F = F^*(H) = Z(H) \circ E_1 \circ \cdots \circ E_m \circ U_1 \circ \cdots \circ U_s,
    \]
    and properties \emph{(i)--(iii)} hold
    (we use the same notation).
    Let $Z = Z(F)$.
    Since $F = F^*(H)$, we have $C_H(F)=Z$ and $H/Z$ is a subdirect subgroup of
    \begin{equation*}
        \prod_{i=1}^m \frac{H}{C_H(E_i)} \times \prod_{j=1}^s \frac{H}{C_H(U_j)}.
    \end{equation*}

    Consider an index $i \in \{1, \dots, m\}$.
    We have a representation
    \[
        \rho_i : H / C_H(E_i) \to \Sp_{2n_i}(p_i)
    \]
    with abelian kernel and completely reducible image, where $p_i \ne p$.
    Moreover we have $2n_i \le p_i^{n_i} \le n/r \le n$ and
    \[
        p_i^{2n_i} < p^{p_i^{n_i}} \le p^n,
    \]
    unless $p = 2$ and $p_i^{n_i} = 3^1$, a case we may ignore as $\Sp_2(3)$ is soluble.
    Thus either $2n_i < n$ or $2n_i = n$ and $p_i < p$, so induction applies to $\rho_i(H / C_H(E_i)))$ and
    \[
        \muna(H / C_H(E_i))% = \muna(\rho_i(H / C_H(E_i)))
        \le b_1^{-1} p_i^{2c_1 n_i}
        < b_1^{-1} p^{c_1 p_i^{n_i}}.
    \]

    Now consider an index $j \in \{1, \dots, s\}$.
    We have
    \[
        T_j^{u_j}\le H/C_H(U_j)\le \Aut(T_j)\wr S_{u_j}
    \]
    and there is a nontrivial irreducible representation
    \(
        Q_j \to \SL_{t_j}(p).
    \)
    In particular $\mu(T_j) < p^{t_j}$ by \cite{kleidman--liebeck}*{Proposition 5.2.1(i)} (see also \Cref{lem:simple-sections-sym-n}).
    Note also that $u_j < t_j^{u_j} \le n$.
    Now exactly as in the proof of \emph{(1)} we have
    \[
        \muna(H / C_H(U_j))
        \le \mu(T_j)^{u_j} b_1^{u_j-1}
        < p^{t_j u_j} b_1^{u_j - 1}.
    \]
    Now we check that
    \[
        p^{t_j u_j} b_1^{u_j - 1} \le b_1^{-1} p^{c_1 t_j^{u_j}},
    \]
    unless $(t_j, p) \in \{(2, 2), (2, 3), (3, 2), (4, 2)\}$.
    The first two cases we may ignore on the grounds that $\SL_{t_j}(p)$ is soluble.
    In the other two cases we may check directly that
    \[
        \mu(T_j) \le b_1^{-1} p^{c_1 t_j}, \qquad \mu(T_j)^{u_j} b_1^{u_j-1} \le b_1^{-1} p^{c_1 t_j^{u_j}}
    \]
    (equality holds if $T_j = \SL_3(2)$ and $u_j = 1$).
    Thus in all cases
    \[
        \muna(H / C_H(U_j)) \le b_1^{-1} p^{c_1 t_j^{u_j}}.
    \]

    Finally, note that $m + s \ge 1$ unless $H$ is abelian.
    Thus by \cref{lem:mucf-subdirect} and the property that $\prod_{i=1}^m p_i^{d_i} \prod_{j=1}^s t_j^{u_j} \le n/r$ we have
    \begin{align*}
        \muna(G) = \muna(H / Z)
        &\le \prod_{i=1}^m \muna(H / C_H(E_i)) \prod_{j=1}^s \muna(H / C_H(U_j))
        \\
        &\le b_1^{-1} p^{c_1 \sum_{i=1}^m p_i^{d_i} + c_1 \sum_{j=1}^s t_j^{u_j}}
        \\
        &\le b_1^{-1} p^{c_1 n},
    \end{align*}
    as claimed.
\end{proof}

We can now complete the proof of \Cref{thm:gen-palfy-wolf}, as well as \Cref{cor:explicit-BCP}.

\begin{proof}[Proof of \Cref{thm:gen-palfy-wolf}]
    Combine \eqref{eq:muab-muna}, \eqref{eq:pyber}, and \Cref{thm:muna-bounds}\emph{(1)}.
\end{proof}

\begin{proof}[Proof of \Cref{cor:explicit-BCP}]
    From \cite{ls74} and the Classification of Finite Simple Groups, there is a constant $c > 0$ such that $\mu(T) \ge |T|^{c/r}$ for every nonabelian finite simple group of rank $r$.
    If $G \le S_n$ is a primitive permutation group of degree $n$ with no composition factors of rank greater than $r \ge 1$, it follows that
    \[
    |G|^{c/r} \le \mucf(G) < n^5
    \]
    by \Cref{thm:gen-palfy-wolf}.
    Hence $|G| < n^{Cr}$ where $C = 5/c$.
\end{proof}

As an immediate consequence of \Cref{thm:gen-palfy-wolf} and Schreier's lemma (\Cref{lem:schreier}) we can bound the diameter of a primitive permutation group $G \le S_n$ in terms of $\theta_1(G)$.
This is a special case of \Cref{cor:transitive},
and we will use this special case in the next section.
The proof of the general case of \Cref{cor:transitive} is deferred to \Cref{sec:perm-groups}.

\begin{corollary}
    \label[corollary]{thm:primitive}
    If $G \le S_n$ is primitive then
    \(
        \diam(G) < \tfrac14 n^{10 + 5 \theta_1(G)}.
    \)
\end{corollary}

\begin{proof}
    Let $G \le S_n$ be a primitive group and let
    \(
    G = G_0 > \cdots > G_\ell = 1
    \)
    be a composition series for $G$.
    Let $T_i = G_{i-1} / G_i$ for $1 \le i \le \ell$.
    Let $\theta_1 = \theta_1(G)$.
    Then
    \[
    \diam(T_i) \le \mu(T_i)^{\theta_1} \qquad (1 \le i \le l).
    \]
    Indeed, this is immediate from the definition if $T_i$ is nonabelian, and if $T_i$ is abelian then it holds because $\mu(T_i) = |T_i|$ and $\theta_1 \ge 1$.
    By \Cref{lem:schreier},
    \begin{align*}
        \diam(G)
        \le 4^{\ell-1} \prod_{i=1}^{\ell} \diam(T_i)
        \le 4^{\ell-1} \prod_{i=1}^\ell \mu(T_i)^{\theta_1}
        = 4^{\ell-1} \mucf(G)^{\theta_1}.
    \end{align*}
    By \Cref{thm:gen-palfy-wolf}, $\mucf(G) < n^5$
    and $\ell < 5 \log_2(n)$.
    It follows that
    \[
    \diam(G) < 4^{5 \log_2(n) - 1} n^{5\theta_1} \le \tfrac14 n^{10 + 5\theta_1},
    \]
    as claimed.
\end{proof}

\section{Groups with trivial soluble radical}
\label[section]{sec:anabelian-socle}

If $G$ is a finite group we denote by $\sol(G)$ the soluble radical of $G$, i.e., the largest soluble normal subgroup of $G$.
In this section we consider finite groups $G$ with $\sol(G)$ trivial.
The \emph{socle} of $G$ is the subgroup $\soc(G)$ generated by the minimal normal subgroups.
If $\sol(G) = 1$ then $\soc(G)$ is a direct product of nonabelian simple groups.

A critical subcase is the case in which $G$ is monolithic with nonabelian socle.
A finite group $G$ is said to be \emph{monolithic} if it has a unique minimal normal subgroup.
A key motivating example is the iterated wreath product
\[
    G_h = A_5 \wr \cdots \wr A_5 \le \Sym(5^h).
\]

The following lemma is basic and well-known.

\begin{lemma}
    \label[lemma]{lem:min-normal-nonabelian}
    Let $N_1, \dots, N_k$ be distinct nonabelian minimal normal subgroups of a group $G$.
    Then
    \begin{enumerate}
        \item $\gen{N_1, \dots, N_k} = N_1 \times \cdots \times N_k$,
        \item every $G$-normal subgroup of $N_1 \times \cdots \times N_k$ has the form $\prod_{i \in I} N_i$.
    \end{enumerate}
    In particular this holds if $N_1, \dots, N_k$ are nonabelian simple groups and $G = \gen{N_1, \dots, N_k}$.
\end{lemma}
\begin{proof}
    Part \emph{(1)} is proved by induction on $k$, using the fact that $[N_i, N_j] \le N_i \cap N_j = 1$ for $i \ne j$. This implies that $N_k$ centralizes $\gen{N_i : 1 \le i < k}$, so $N_k \cap \gen{N_i : 1 \le i < k} = 1$. This implies that
    \[
        \gen{N_1, \dots, N_k} = \gen{N_1, \dots, N_{k-1}} \times N_k = N_1 \times \cdots \times N_k
    \]
    by induction.

    For part \emph{(2)}, suppose $N \le N_1 \times \cdots \times N_k$ and $N \nsgp G$.
    Let $I = \{i : N_i \le N\}$. By quotienting out $\prod_{i\in I} N_i$ we may assume $I = \emptyset$.
    Then $[N, N_i] \le N \cap N_i = 1$ for each $i$, so $N \le Z(\prod N_i) = 1$.
\end{proof}

We now give a series of results that, given $G = \gen X$ and a normal subgroup $N$ of $G$, try to locate a short nontrivial element in $N$.
The subgroup length $\lambda(N)$ was defined in \Cref{sec:soluble},
and recall that $B_X(r) = (X \cup X^{-1} \cup \{1\})^r$ denotes the ball of radius $r$.
A subgroup $N \le M$ is called \emph{self-centralizing} if $C_M(N) \le N$.
The following lemma generalizes \cite{wilson2}*{Lemma~3.4}.

\begin{lemma}
    \label[lemma]{lem:wilson-commutator-argument}
    Let $k, m \ge 1$ and let $M_1, \dots, M_k$ be groups such that, for each $i$,
    $M_i$ has a self-centralizing minimal normal subgroup $N_i$ and $\sublen(N_i) \le m$.
    Let $M = \gen X \le M_1 \times \cdots \times M_k$ be a subdirect subgroup, and let $\pi_i : M \to M_i$ be the natural epimorphism.
    Let $r \ge 1$ and assume that
    \[
        \pi_i(B_X(r)) \cap N_i \ne \{1\} \qquad (1 \le i \le k).
    \]
    Then
    \[
        B_X(4 k^2 (r + 2m)) \cap (N_1 \times \cdots \times N_k) \ne \{1\}.
    \]
\end{lemma}
\begin{proof}
    We will show by induction on $k$ that
    \[
            B_X(a_k) \cap (N_1 \times \cdots \times N_k) \ne \{1\}.
    \]
    where the integer sequence $(a_k)_{k \ge 1}$ is defined by the recurrence equation
    \[
        a_1 = r, \qquad a_k = 4 a_{\ceil{k/2}} + 4m \qquad (k \ge 2).
    \]
    It is straightforward to check by induction that
    \[
        a_k = 4^{i} (r + \tfrac43 m) - \tfrac 43m
    \]
    where $i = \ceil{\log_2 k}$ is the unique integer such that $2^{i-1} < k \le 2^i$; indeed, it suffices to note that $2^{i-2} < \ceil{k/2} \le 2^{i-1}$.
    This will prove the lemma since
    \[a_k < 4 k^2 (r + 2m).\]
    Note also that $a_k$ is monotonic in $k$: this will be used in the proof.

    The case $k = 1$ is trivial, so assume $k \ge 2$.
    For any subset $I \subseteq \{1, \dots, k\}$, let $\pi_I$ denote the projection map
    \[\pi_I: M \to M_I = \prod_{i\in I} M_i.\]
    Let $I \subset \{1, \dots, k\}$ be a subset of cardinality $k_1 = \ceil{k/2} < k$, and consider $\pi_I(M) = \gen{\pi_I(X)} \le M_I$.
    By induction, there is an element $x \in B_X(a_{k_1})$ such that
    \[
        1 \ne \pi_I(x) \in \prod_{i \in I} N_i.
    \]
    Let
    \(
        J = \{j \in \{1, \dots, k\} : \pi_j(x) \notin N_j\}.
    \)
    Then $I \cap J = \emptyset$, so $|J| \le \floor{k/2} \le k_1$.
    If $J = \emptyset$ then we may take $g = x$.
    Otherwise, by induction there is $y \in B_X(a_{k_1})$ such that
    \[
        1 \ne \pi_J(y) \in \prod_{j \in J} N_j.
    \]
    Let $j \in J$ be an index such that $\pi_j(y) \ne 1$.
    Then
    \[
        \pi_j(x) \in M_j \setminus N_j , \qquad \pi_j(y) \in N_j \setminus \{1\}.
    \]
    Since $N_j$ is minimal normal in $M_j$ and $\pi_j(y) \ne 1$,
    $N_j$ is the normal closure of $\pi_j(y)$, so \Cref{lem:generalized-Milnor} gives
    \[
        N_j = \gen{\pi_j(y^h) : h \in B_X(m)}.
    \]
    Since $C_{M_j}(N_j) \le N_j$ and $\pi_j(x) \notin N_j$, there is some $h \in B_X(m)$ such that
    \[
        \pi_j([x, y^h]) = [\pi_j(x), \pi_j(y^h)] \ne 1.
    \]
    Let $g = [x, y^h]$. Then $g$ is a nontrivial element of $N_1 \times \cdots \times N_k$ such that
    \[
        \len_X(g)
        \le 2 \len_X(x) + 2 \len_X(y) + 4 m
        \le 4 a_{k_1} + 4m = a_k.
    \]
    This finishes the proof.
\end{proof}

\begin{proposition}
    \label[proposition]{prop:monolithic-nonabelian-socle}
    Let $G = \gen X$ be a finite monolithic group with nonabelian socle $N \cong T^n$, where $T$ is a nonabelian simple group and $n\ge 1$.
    Then
    \[
        B_X(2n^{16 + 5 \theta_1(G)} \log_2 |T|) \cap N \ne \{1\}.
    \]
\end{proposition}

\begin{proof}
    We use induction on $n$.
    If $n = 1$ then $T \le G \le \Aut(T)$.
    By \Cref{lem:schreier}, $N = \gen Z$ where $\len_X(Z) \le 2 |\Out(T)| + 1$.
    We have $|\Out(T)| \le \log_2 |T| - 1$ (see \cite{Kohl} or \cite{Menezes}*{Lemma 2.3.2}),
    so indeed we have $\len_X(Z)\le 2\log_2 |T|$, as required.

    Hence assume $n \ge 2$.
    Then $G \le \Aut(N) \cong \Aut(T) \wr S_n$.
    Let $\phi : G \to S_n$ be the induced homomorphism.
    Since $N$ is minimal normal, $\phi(G) \le S_n$ is transitive.
    Let $K = \ker \phi = G \cap \Aut(T)^n$, so $G / K \cong \phi(G)$.
    Let $M / K$ be a maximal normal intransitive subgroup of $G / K$.

    Let $N_1$ be a minimal normal subgroup of $M$.
    Let the $G$-conjugates of $N_1$ be denoted $N_1, \dots, N_k$.
    Then, by \Cref{lem:min-normal-nonabelian},
    \[
        N = \gen{N_1^G} = N_1 \times \cdots \times N_k = T_1 \times \cdots \times T_n,
    \]
    and each $N_j$ is of the form $\prod_{i \in I} T_i$ where $I \subset \{1, \dots, n\}$ is an orbit of $\phi(M)$.
    Also, $k > 1$ since $\phi(M)$ is intransitive.

    Now $G / M$ permutes $N_1, \dots, N_k$ transitively.
    Let $\Delta = \{B_1, \dots, B_s\}$ be a system of maximal blocks for this action, where $1 < s \mid k$.
    If $M_1$ is the kernel of the action $G \to \Sym(\Delta)$ then $M \le M_1 \nsgp G$ and $\phi(M_1)$ is intransitive, so $M = M_1$ by maximality of $M$.
    Thus $G / M$ has a faithful primitive action of degree $s \le k$.
    Applying \Cref{thm:primitive},
    \[
        \diam(G / M) < \tfrac14 k^{10 + 5 \theta_1},
    \]
    where $\theta_1 = \theta_1(G)$.
    By \Cref{lem:schreier} it follows that $M = \gen Y$ where
    \[
        \len_X(Y) \le k^{10 + 5 \theta_1}.
    \]

    Let $M_i = M / C_M(N_i)$ for $i = 1, \dots, k$.
    Then $M_i$ is monolithic with socle $N_iC_M(N_i)/C_M(N_i) \cong N_i \cong T^{n/k}$,
    and $M = \gen Y$ is subdirect in $M_1 \times \cdots \times M_k$ since
    \(
        \bigcap_{i=1}^m C_M(N_i) = C_M(N) = 1.
    \)
    Write $\pi_i:M\rightarrow M_i$ for the natural projection.
    By induction,
    \[
        \pi_i(B_Y(r)) \cap N_i \ne \{1\},
    \]
    where $r = 2 (n/k)^{16 + 5 \theta_1} \log_2 |T|$.
    Now by \Cref{lem:wilson-commutator-argument}
    and the bound
    \[
        \sublen(N_i) \le \log_2 |N_i| \le (n/k) \log_2 |T| \le \tfrac12 r,
    \]
    there is a nontrivial element $g \in N = N_1 \times \cdots \times N_k$ such that
    \begin{align*}
        \len_Y(g)
        \le 4 k^2 (r + 2m)
        \le 2^4 k^2 (n / k)^{16 + 5 \theta_1} \log_2 |T|.
    \end{align*}
    This implies that
    \begin{align*}
        \len_X(g) \le \len_X(Y) \len_Y(g)
        &\le 2^4 k^{12 + 5 \theta_1} (n/k)^{16 + 5 \theta_1} \log_2 |T| \\
        &= 2^4 k^{-4} n^{16 + 5 \theta_1} \log_2 |T| \\
        &\le n^{16 + 5 \theta_1} \log_2 |T|
    \end{align*}
    since $k \ge 2$.
    This concludes the proof.
\end{proof}

The following proposition is the main result of this section.

\begin{proposition}
    \label[proposition]{thm:anabelian-socle}
    Let $G = \gen X$ be a nontrivial finite group with trivial soluble radical.
    Let $S = \soc(G)$.
    Then
    \[
        B_X((\log |S|)^{18 + 5\theta_1(G)}) \cap S \ne \{1\}.
    \]
\end{proposition}

\begin{proof}
    Let $N_1, \dots, N_k$ be the minimal normal subgroups of $G$.
    By \Cref{lem:min-normal-nonabelian}, $S = N_1 \times \cdots \times N_k$.
    Let $G_i = G / C_G(N_i)$ for $1 \le i \le k$.
    Then each $G_i$ is monolithic with socle $N_i \cong T_i^{n_i}$ say and $G$ is a subdirect in $G_1 \times \cdots \times G_k$.
    Let $\pi_i$ be the projection $G \to G_i$.
    Then $G_i = \gen{\pi_i(X)}$.
    By \Cref{prop:monolithic-nonabelian-socle}, we have
    \[
        \pi_i(B_X(r)) \cap N_i \ne \{1\} \qquad (1 \le i \le k),
    \]
    where
    \[
        r = \max_{1 \le i \le k} 2 n_i^{16 + 5 \theta_1} \log_2|T_i|, \qquad \theta_1 = \theta_1(G).
    \]
    Let $m = \max_{1 \le i \le k} \log_2|N_i|$.
    Then applying \Cref{lem:wilson-commutator-argument} gives
    \[
        B_X(4 k^2 (r + 2m)) \cap (N_1 \times \cdots \times N_k) \ne \{1\}.
    \]
    Bounding crudely,
    \begin{align*}
        k &\le \log_{60} |S| = a^{-1} \log |S|, \qquad\text{where}~a = \log (60) \approx 4.1, \\
        r &\le 2 (\log_{60} |S|)^{15 + 5 \theta_1} \log_2 |S| \le 4 a^{-20} (\log |S|)^{16 + 5\theta_1}, \\
        m &\le \log_2 |S| \le 2a^{-20} (\log |S|)^{16 + 5\theta_1},
    \end{align*}
    so
    \[
        4 k^2 (r + 2m) < (\log |S|)^{18 + 5 \theta_1}.\qedhere
    \]
\end{proof}

To apply \Cref{thm:anabelian-socle} effectively,
we need to bound the diameter of a direct product of simple groups.
The following result is taken from Dona~\cite{dona-direct-products}.

\begin{theorem}
    \label[theorem]{thm:dona}
    Let $G = T_1 \times \cdots \times T_n$ where $T_1, \dots, T_n$ are nonabelian finite simple groups of rank at most $r$.
    Then
    \[
        \diam(G) \ll n^3 r \max_{1 \le i \le n} \diam(T_i).
    \]
\end{theorem}

\begin{remark}
    This result was first proved by Babai--Seress in \cite{BS92}*{Lemma~5.4} in the case of alternating groups,
    and later generalized to all finite simple groups by Dona~\cite{dona-direct-products}.
    See the comments in \cite{helfgott18}*{Section~4} for more commentary.
    We give a proof below only to demonstrate that the use of Ore's conjecture can be avoided.
    The proof is similar to that of \Cref{lem:wilson-commutator-argument}.
\end{remark}

\begin{proof}
    Let $d = \max_{1 \le i \le n} \diam(T_i)$.
    Say $G = \gen X$. Fix an index $i \in \{1, \dots, n\}$, without loss of generality $i = 1$, and for any $J \subset \{2, \dots, n\}$ let $\pi_J$ denote the projection $G \to T_1 \times \prod_{j \in J} T_j$.
    For each index $j \ne 1$, Schreier's lemma (\Cref{lem:schreier}) implies that $\pi_{\{j\}}(B_X(2d+1))$ contains a nontrivial element of $T_1 \times 1$.
    By conjugation, it follows that $\pi_{\{j\}}(B_X(4d+1))$ contains a nontrivial conjugacy class of $T_1 \times 1$.

    We claim by induction that for each subset $J \subset \{2, \dots, n\}$ of cardinality $|J| \le 2^k$, the set $\pi_J(B_X(4^k (4d+1)))$ contains a nontrivial conjugacy class in $T_1 \times 1^J$.
    This holds for $k = 0$ by the previous paragraph, so assume $k > 0$
    and let $J \subseteq \{2, \dots, n\}$ be a set of cardinality $|J| \le 2^k$.
    Partition $J$ into two sets $J_1$ and $J_2$ each of cardinality at most $2^{k-1}$.
    By induction, $\pi_{J_i}(B_X(4^{k-1}(4d+1)))$ contains a nontrivial conjugacy class of $T_1 \times 1^{J_i}$ for $i = 1, 2$.
    By taking commutators, it follows that $\pi_J(B_X(4^k(4d+1)))$ contains a nontrivial conjugacy class of $T_1 \times 1^J$, as claimed.

    Applying the inductive claim to $J = \{2, \dots, n\}$, it follows that $B_X(20 n^2 d)$ contains a nontrivial conjugacy class of $T_1$.
    By \cite{dona-direct-products}*{Proposition~2.4}, $T_1$ has diameter at most $8(5r+7) \le 100r$ with respect to any nontrivial conjugacy class, so $\len_X(T_1) \le 2000 n^2 r d$.
    The index $i = 1$ was arbitrary, so $\len_X(T_i) \le 2000 n^2 r d$ for $1 \le i \le n$, and it follows that $\len_X(G) \le 2000 n^3 r d$.
\end{proof}

\begin{remark}
    Determining the order of growth of the sequence $\diam(A_5^n)$ is a fascinating open problem.
    We have $n \ll \diam(A_5^n) \ll n^3$.
\end{remark}

At this point we can easily complete the proof of the main theorem in the anabelian case.

\begin{theorem}
    \label[theorem]{thm:anabelian-case}
    Let $G \ne 1$ be a finite anabelian group. Then
    \[
        \diam(G)
        \ll (\log|G|)^{23 + 5 \theta_1(G)}
        \max_{T \in \nacomp(G)} \diam(T).
    \]
\end{theorem}
\begin{proof}
    Let $\theta_1 = \theta_1(G)$ and $d = \max_{T \in \nacomp(G)} \diam(T)$.
    Let $S = \soc(G)$.
    Suppose that $G = \gen X$.
    By \Cref{thm:anabelian-socle}, there is a nontrivial element $g \in S$ such that
    \[
        \len_X(g) \le (\log |S|)^{18 + 5 \theta_1}.
    \]
    Let $N = \gen{g^G}$. Then $N = \gen{Y}$ where
    \[
        Y = \{g^h : \len_X(h) \le \log_2 |S|\},
        \qquad \len_X(Y) \ll (\log |S|)^{18 + 5 \theta_1}.
    \]
    Also $N$ is a normal subgroup of $G$ contained in $S = \soc(G)$, so by \Cref{lem:min-normal-nonabelian} we have $N = \prod_{i=1}^n T_i$ for some nonabelian simple groups $T_1, \dots, T_n$.
    Applying \Cref{thm:dona},
    \[
        \len_X(N) \le \len_X(Y) \diam(N) \ll (\log |S|)^{18 + 5 \theta_1} \cdot n^3 r d,
    \]
    where $r$ is the maximum rank of $T_1, \dots, T_n$. Crudely bounding
    \[
        n \ll \log |S|,
        \qquad
        r \ll \log |S|,
    \]
    we have
    \[
        \len_X(N) \ll (\log |S|)^{22 + 5 \theta_1} d
        \ll (\log |G|)^{22 + 5 \theta_1} d.
    \]

    Now we use
    \[
        \len_X(G) \le \len_X(N) + \len_X(G/N) \le \len_X(N) + \diam(G/N)
    \]
    (from \Cref{lem:extension-rule})
    and induction. We get that
    \[
        \len_X(G) \ll m (\log |G|)^{22 + 5 \theta_1} d
    \]
    where $m$ is the composition length of $G$. Since $m \le \log_{60} |G|$, we are done.
\end{proof}

\section{Arbitrary finite groups}

In this section we consider arbitrary finite groups by combining our results for soluble groups and for groups with trivial soluble radical.

\begin{lemma}
    \label[lemma]{lem:inductive-lemma}
    Let $G = \gen X$ be a finite group.
    If $G$ is insoluble then there is an insoluble normal subgroup $N \nsgp G$ such that
    \[
        \len_X(N) \ll \eps(G) (\log |G|)^{30 + 5 \theta_1(G)} \max_{T \in \nacomp(G)} \diam(T).
    \]
\end{lemma}
\begin{proof}
    Define $\eps = \eps(G)$, $\theta_1 = \theta_1(G)$, and
    \[d = \max_{T \in \nacomp(G)} \diam(T).\]
    Let $G^* = G / \sol(G)$ and let $g \mapsto g^*$ be the natural quotient map.
    Applying \Cref{thm:anabelian-socle} to $G^* = \gen{X^*}$ and $S = \soc(G^*)$, there is an element $g \in G$ with $1 \ne g^* \in S$ and
    \[
        \len_X(g) = \len_{X^*}(g^*) \le (\log |S|)^{18 + 5 \theta_1}.
    \]
    Let $N = \gen{g^G} \nsgp G$.
    By \Cref{lem:generalized-Milnor}, $N = \gen Z$ where
    \begin{align*}
        \len_X(Z)
        \le \len_X(g) + 2 \sublen(N)
        &\ll (\log |S|)^{18 + 5 \theta_1} + \log |\sol(G)| \\
        &\ll (\log |G|)^{18 + 5 \theta_1}.
    \end{align*}
    Now $1 \ne N^* \nsgp S$, so $N^*$ is a nontrivial direct product of nonabelian composition factors of $G$.
    By \Cref{thm:dona}, it follows that
    \[
        \diam(N^*) \ll (\log |N^*|)^4 d \le (\log |G|)^4 d.
    \]
    On the other hand $N \cap \sol(G) = \sol(N)$, so $N^* \cong N / \sol(N)$, and by \Cref{thm:main-soluble} we have
    \[
        \diam(\sol(N)) \ll \eps_0(\sol(N)) (\log |\sol(N)|)^8 \ll \eps (\log |G|)^8.
    \]
    Thus by Schreier's lemma (\Cref{lem:schreier}) we have
    \begin{align*}
        \diam(N)
        &\ll \diam(\sol(N)) \diam(N / \sol(N)) \\
        &\ll \eps (\log |G|)^{12} d.
    \end{align*}
    Since $N = \gen Z$, it follows that
    \[
        \len_X(N) \le \len_X(Z) \diam(N)
        \ll \eps (\log|G|)^{30 + 5\theta_1} d,
    \]
    as claimed.
\end{proof}

We can now give the proof of \Cref{thm:main} in complete generality.

\begin{proof}[Proof of \Cref{thm:main}]
    Let $G = \gen X$ be a finite group.
    Define $\eps = \eps(G)$, $\theta_1 = \theta_1(G)$, and
    \[
        d = \max\{\diam(T) : T \in \nacomp(G)\} \cup \{1\}.
    \]
    If $G$ is insoluble then \Cref{lem:inductive-lemma} assures the existence of an insoluble normal subgroup $N \nsgp G$ such that
    \begin{align*}
        \len_X(G)
        &\le \len_X(G/N) + \len_X(N)
        \le \len_X(G/N) + c \eps (\log |G|)^{30 + 5 \theta_1} d,
    \end{align*}
    where $c > 0$ is some constant.
    If $G/N$ is again insoluble, a similar bound applies to $\len_X(G/N)$, and so on.
    If $G$ has $m$ nonabelian composition factors, it follows by induction that there is a normal subgroup $N \nsgp G$ with $G/N$ soluble such that
    \[
        \len_X(G) \le \len_X(G/N) + c m \eps (\log |G|)^{30 + 5 \theta_1} d.
    \]
    Note that $m \le \log_{60}|G| \le \log |G|$, while by \Cref{thm:main-soluble} we have
    \[
        \len_X(G/N) \le \diam(G/N) \ll \eps_0(G/N) (\log |G/N|)^8 \le \eps (\log |G|)^8.
    \]
    Thus
    \[
        \len_X(G) \ll \eps (\log |G|)^8 + \eps (\log |G|)^{31 + 5\theta_1} d.
    \]
    This completes the proof of the theorem.
\end{proof}

\section{Permutation groups}
\label[section]{sec:perm-groups}

To apply our results to permutation groups, we need to bound the exponent parameter $\eps(G)$.
For this we have the following convenient lemma, which follows easily from results of Guralnick~\cite{guralnick}.

\begin{lemma}
    \label[lemma]{lem:guralnick}
    Let $G \le \Sym(n)$ be a transitive group and let $N \nsgp G$.
    Assume the orbits of $N$ have length $m$.
    Then $\exp(N / N') \le m$.
    In particular $\eps(G) \le n$.
\end{lemma}
\begin{proof}
    Let $R$ be the transitive group induced by $N$ on one of its orbits. Then $N$ is subdirect in $R^{n/m}$.
    Let $\gamma(H)$ denote the group invariant introduced in \cite{guralnick}*{Section~2} (denoted there by $\mu(H)$).
    By \cite{guralnick}*{Lemmas~2.1, 2.3, and Theorem~4.1(a)}, we have
    \[
    \exp(N/N') = \gamma(N/N') \le \gamma(N) = \gamma(R) \le m.\qedhere
    \]
\end{proof}

The following lemma is well-known; see Kov\'{a}cs--Praeger~\cite{KP00} for a much more general result.

\begin{lemma}
    \label[lemma]{lem:simple-sections-sym-n}
    If $T$ is simple and $T \cong G / N$ with $N \nsgp G \le S_n$ then $\mu(T) \le n$.
\end{lemma}
\begin{proof}
    Let $n$ be the smallest integer such that $T \cong G / N$ for some subgroups $N \nsgp G \le S_n$.
    If $G_x N = G$ then $G_x / G_x \cap N \cong G_x N / N = G / N \cong T$, contradicting the minimality of $n$, so $G_x N < G$.
    It follows that $G_x N / N$ is a proper subgroup of $G / N \cong T$ of index $|G : G_x N| \le |G : G_x| \le n$.
\end{proof}

We can now prove the main corollary listed in the introduction on transitive permutation groups.

\begin{proof}[Proof of \Cref{cor:transitive}]
    Let $G \le S_n$ be transitive. Let $\eps = \eps(G)$ and $\theta_1 = \theta_1(G)$. By \Cref{thm:main},
    \[
        \diam(G) \ll \eps (\log |G|)^{31 + 5 \theta_1} \diam(T)
    \]
    for some $T \in \nacomp(G)$, or else the same without the $\diam(T)$ factor if $\nacomp(G) = \emptyset$.
    By \Cref{lem:guralnick}, $\eps \le n$, and $\log|G| \le n \log n$.
    By \Cref{lem:simple-sections-sym-n}, $\mu(T) \le n$, so $\diam(T) \le n^{\theta_1}$.
    Thus
    \[
        \diam(G) \ll \eps (n \log n)^{31 + 5 \theta_1} n^{\theta_1}
        \ll n^{32 + 7 \theta_1}.
    \]
    This proves the first statement.

    Now assume that the following two statements hold:
    \begin{enumerate}[(a)]
        \item $\diam(A_n) \le n^{O(1)}$,
        \item $\diam(G) \le |V|^{O(1)}$ for simple classical groups $G \le \PGL(V)$.
    \end{enumerate}
    In the latter case $\mu(G) \gg |V|^{1/2}$ (see \cite{kleidman--liebeck}*{Theorem~5.2.2}), while obviously $\mu(A_n) = n$, so the two statements imply $\theta_1(T) \ll 1$ for all finite simple groups of rank larger than $8$.
    Since Babai's conjecture is known for simple groups of bounded rank, it follows that $\theta_1(T) \ll 1$ for all simple groups,
    which implies that $\theta_1(G) \ll 1$ for all finite groups.
    Thus \Cref{conj:folk} holds.
\end{proof}

\begin{proof}[Proof of \Cref{cor:transitive-2}]
    By \Cref{cor:transitive} we have $\diam(G) \ll n^{32 + 7\theta_1}$ where $\theta_1 = \theta_1(G)$.
    Assuming $\theta_1 > 1$, let $T \in \nacomp(G)$ realize the maximum in the definition of $\theta_1(G)$.
    If $T$ has bounded rank then $\theta_1 \ll 1$.
    If $T \cong A_m$ for some $m \ge 5$ then
    \[\theta_1 \ll (\log m)^3 \log \log m\]
    by \Cref{eq:helfgott--seress}.
    If $T \le \PGL(V)$ is classical then by \Cref{lem:simple-sections-sym-n} and \cite{kleidman--liebeck}*{Theorem~5.2.2} again we have $n \ge \mu(T) \gg |V|^{1/2}$, so
    \[\theta_1(T) \ll (\log \dim V)^2 \ll (\log \log n)^2\]
    by \eqref{eq:HMPQ}.
\end{proof}

In particular if $G \le S_n$ is transitive and soluble then $\theta_1(G) = 1$, so \Cref{cor:transitive} gives
\(
    \diam(G) \ll n^{39}.
\)
In fact we can obtain
\[
    \diam(G) \le n^{6.16}
\]
by combining the bound $\diam(G) \le \eps(G) 4^{L-1} (\log_2|G|)^2$ from \Cref{thm:main-soluble} with \Cref{lem:guralnick} and the following standard bounds for order and the derived length of a transitive soluble group:
\[
    |G| \le 24^{(n-1)/3}, \qquad
    L \le (5/2) \log_3 n.
\]
(see \cite{DM}*{Theorem~5.8B and Exercise~5.8.7}).

Let us now invest some energy to reduce the exponent in the result for transitive soluble groups.
As usual, $\gamma_3(G)$ denotes the subgroup generated by the triple commutators $[x, y, z] = [[x,y],z]$ with $x, y, z \in G$.
If $G = \gen X$ then $\gamma_3(G)$ is normally generated by the triple commutators $[x, y, z]$ with $x, y, z \in X$.

\begin{lemma}
    [Variant of \Cref{lem:gen-Milnor-N'}]
    \label[lemma]{lem:gen-Milnor-gamma_3}
    Let $G = \gen X$ and let $N = \gen Y \nsgp G$ be a normal subgroup with $\gamma_3(N)$ finite. Then $\gamma_3(N) = \gen Z$ where
    \[
        \len_X(Z) \le 10\len_X(Y) + 2 \sublen(\gamma_3(N)).
    \]
\end{lemma}
\begin{proof}
    Let $Z_0 = \{[y_1, y_2, y_3] : y_1,y_2,y_3 \in Y\}$. The triple commutator word $[x,y,z]$ has length $10$, so $\len_X(Z_0) \le 10 \len_X(Y)$, and $Z_0$ normally generates $\gamma_3(N)$. Apply \Cref{lem:generalized-Milnor} to $\gamma_3(N) = \gen{Z_0^G}$.
\end{proof}

\begin{proof}[Proof of \Cref{cor:soluble-transitive}]
    Let $G \le S_n$ be soluble and transitive,
    and let $H$ be a maximal soluble subgroup of $S_n$ containing $G$. Then $H = P_h \wr \cdots \wr P_1$ for some primitive soluble groups $P_j \le \Sym(d_j)$ where $n = d_1 \cdots d_h$.
    By \cite{EM25}*{Theorem~4.6}(i), for each $j$ there is a series of subgroups
    \[
        P_j = P_{j,0} > P_{j,1} > \cdots > P_{j,\ell_j} = 1
    \]
    with the following properties:
    \begin{enumerate}
        \item for each $i = 0, 1, \dots, \ell_j - 1$, $P_{j,i+1}$ is either $(P_{j,i})'$ or $\gamma_3(P_{j,i})$,
        \item if $m_j$ is the number of nonabelian factors $P_{j,i} / P_{j,i+1}$ then
        \[
            \ell_j + m_j \log_4(5/2) \le 3 \log_4(d_j),
        \]
        \item $\ell_j \ll 1 + \log \log_2 d_j$ by Zassenhaus--Newman~\cite{newman}.
    \end{enumerate}
    By stringing these series together for $j = 1, \dots, h$ we get a series
    \[
        G = G_0 > G_1 > \cdots > G_L = 1
    \]
    with the following properties:
    \begin{enumerate}[resume]
        \item for each index $i = 0, \dots, L-1$, $G_{i+1}$ is either $(G_i)'$ or $\gamma_3(G_i)$,
        \item for each index $i \ge \ell_1 + \cdots + \ell_j$, $G_i \le (P_h \wr \cdots \wr P_{j+1})^{d_1\cdots d_j}$, so the orbits of $G_i$ have length at most $d_{j+1} \cdots d_h = n / d_1 \cdots d_j$.
    \end{enumerate}
    Note in particular that $G_i \nsgp G$ for each $i$.

    Now let $X$ be a generating set for $G$ and define sets $X_i \subseteq G_i$ for $i = 0 \dots, L-1$ inductively as follows, starting with $X_0 = X$.
    If $0 < i < L$ and $G_i = (G_{i-1})'$, apply \Cref{lem:gen-Milnor-N'} to obtain $X_i \subseteq G_i$ such that $G_i = \gen{X_i}$ and
    \[
        \len_X(X_i) \le 4 \len_X(X_{i-1}) + 2 \log_2 |G|.
    \]
    Otherwise $G_i = \gamma_3(G_{i-1})$; apply \Cref{lem:gen-Milnor-gamma_3} to obtain $X_i \subseteq G_i$ such that $G_i = \gen{X_i}$ and
    \[
        \len_X(X_i) \le 10 \len_X(X_{i-1}) + 2 \log_2 |G|.
    \]
    Define $a_i = 4$ in the first case and $a_i = 10$ in the second case.
    Then
    \[
        \frac{\len_X(X_i)}{a_1 \cdots a_i} - \frac{\len_X(X_{i-1})}{a_1 \cdots a_{i-1}} \le \frac2{4^i} \log_2 |G|,
    \]
    and we deduce by summing that
    \[
        \frac{\len_X(X_i)}{a_1 \cdots a_i} \le 1 + \frac23 \log_2 |G| \ll n.
    \]
    Thus
    \[
        \len_X(X_i) \ll (a_1 \cdots a_i) n = 4^{i + \log_4(5/2) r_i} n
    \]
    for $0 \le i < L$, where $r_i$ is the number of nonabelian sections $G_{j-1} / G_j$ with $j \le i$.

    Fix an index $i$ with $0 \le i < L$ and let $j \in \{0, \dots, h-1\}$ be the unique index such that
    \[
        \ell_1 + \cdots + \ell_j \le i < \ell_1 + \cdots + \ell_{j+1}.
    \]
    By property \emph{(5)}, the orbits of $G_i$ have length some
    \[
        n_i \le d_{j+1} \cdots d_h = \frac{n}{d_1 \cdots d_j}.
    \]
    Now by properties \emph{(2)} and \emph{(3)}
    \begin{align*}
        \len_X(X_i)
        \ll 4^{\sum_{s=1}^{j+1} (\ell_s + m_s \log_4(5/2))} n
        \ll (d_1 \cdots d_j)^3 (\log_2 d_{j+1})^{O(1)} n.
    \end{align*}

    By \Cref{lem:guralnick}, $\exp(G_i / G_i') \le n_i$ and also $\exp(G_i' / G_{i+1}) \le n_i$.
    If $G_i / G_{i+1}$ is abelian then
    \[
        \diam(G_i / G_{i+1}) \le \exp(G_i / G_{i+1}) \log_2 |G_i : G_{i+1}|
        \le n_i \log_2 |G_i : G_{i+1}|
    \]
    by \Cref{lem:diam-abelian} as before.
    If $G_i / G_{i+1}$ is class-2 nilpotent then, since the commutator subgroup of a class-2 nilpotent group is generated by the commutators of the generators, by \Cref{lem:extension-rule} we have
    \begin{align*}
        \diam(G_i / G_{i+1})
        &\le \diam(G_i / G_i') + 4 \diam(G_i' / G_{i+1}) \\
        &\le n_i \log_2 |G_i : G_i'| + 4 n_i \log_2 |G_i' : G_{i+1}| \\
        &\le 4 n_i \log_2 |G_i : G_{i+1}|.
    \end{align*}
    Now by \Cref{lem:extension-rule} again we have
    \begin{align*}
        \len_X(G)
        &\le \sum_{i=0}^{L-1} \len_X(X_i) \diam(G_i / G_{i+1}) \\
        &\ll \sum_{i=0}^{L-1} \len_X(X_i) n_i \log_2 |G_i:G_{i+1}| \\
        &\ll n \max_{0 \le i < L} \len_X(X_i) n_i \\
        &\ll n^3 \max_{0 \le j < h} (d_1 \cdots d_j)^2 (\log_2 d_{j+1})^{O(1)} \\
        &\le n^5 \max_{0 \le j < h} d_{j+1}^{-2} (\log_2 d_{j+1})^{O(1)}
        \ll n^5.
    \end{align*}
    This completes the proof.
\end{proof}

\begin{remark}
    A simplified version of the above argument using just the derived series gives the exponent $5.16$ instead of $5$.
    Key motivating examples are the iterated wreath products
    \begin{align*}
        G &= \Sym(4) \wr \cdots \wr \Sym(4) \le \Sym(4^h), \\
        G &= \AGL_2(3) \wr \cdots \wr \AGL_2(3) \le \Sym(9^h)
    \end{align*}
    (the latter is the example achieving Dixon's bound on the derived length of a soluble permutation group).
\end{remark}

In the case of a nilpotent group, an easier argument suffices and obtains the exponent $4$.

\begin{proposition}
    Let $G \le S_n$ be a transitive nilpotent group of degree $n$. Then $\diam(G) \le n^4$.
\end{proposition}
\begin{proof}
    Similar to the previous proof, but easier. Let
    \[
        G = G^{(0)} > G^{(1)} > \cdots > G^{(L)} = 1
    \]
    be the derived series of $G$ and let $n_i$ be the length of the orbits of $G^{(i)}$. Since primitive nilpotent groups are abelian, $n_{i+1} < n_i$ for each $i < L$.
    It follows that $n_i \le n / 2^i$.
    In particular $L \le \log_2 n$.
    Also $\exp(G^{(i)} / G^{(i+1)}) \le n_i$ by \Cref{lem:guralnick}.
    Now suppose $G = \gen X$. By \Cref{lem:Xi}, $G^{(i)} = \gen{X_i}$ where
    \[
        \len_X(X_i) \le 4^i \log_2|G|.
    \]
    Also, it is well-known that $|G| \le 2^{n-1}$, so $\log_2|G| \le n$.
    Hence
    \begin{align*}
        \len_X(G)
        &\le \sum_{i=0}^{L-1} \len_X(X_i) \diam(G^{(i)} / G^{(i+1)}) \\
        &\le n \sum_{i=0}^{L-1} 4^i n_i \log_2|G^{(i)} : G^{(i+1)}| \\
        &\le n^3 \max_{0 \le i < L} 2^i \le n^4.\qedhere
    \end{align*}
\end{proof}

\begin{remark}
    \label[remark]{rem:bartholdi--bradford}
    The argument above for nilpotent groups can be optimized to at least $\ll n^{3.67}$ or so using more careful arguments.
    However, Bartholdi and Bradford (personal communication)
    have results that imply a quadratic bound $\diam(G) \ll n^2$ for transitive nilpotent groups $G \le S_n$.
    This is close to optimal in view of congruence quotients of the Grigorchuk group: see \Cref{sec:examples}.
    Their beautiful argument uses the group ring in the spirit of Grigorchuk's result on residually nilpotent groups~\cite{grig-residually-nilpotent}.
\end{remark}

In the case of primitive groups $G \le S_n$, let us first recall \Cref{thm:primitive}, which was internally crucial in \Cref{sec:anabelian-socle}: if $G \le S_n$ is primitive then
\[
    \diam(G) \le n^{10 + 5 \theta_1(G)}.
\]
In fact if $|G|$ is less than exponential and Babai's conjecture holds for the nonabelian composition factors of $G$ then there is a nearly linear bound.

\begin{proposition}
    If $G \le S_n$ is primitive,
    \[
    \diam(G) \le n (\log |G|)^{O(\theta_2(G))}.
    \]
\end{proposition}
\begin{proof}
    Combine \Cref{cor:theta-2-cor} and \Cref{lem:guralnick}.
\end{proof}

In the case of a primitive soluble group, we have the following explicit bound, which is optimal up to the logarithmic factors, considering for example $G = C_p$ or $G = \AGL_1(p)$.

\begin{proposition}
    Let $G \le S_n$ be soluble and primitive. Then
    \[
        \diam(G) \ll n (\log n)^8.
    \]
\end{proposition}
\begin{proof}
    By the P\'alfy--Wolf theorem, $|G| \le n^{O(1)}$, so $\log_2 |G| \ll \log n$.
    Now apply \Cref{thm:main-soluble,lem:guralnick}.
\end{proof}

\section{Residually finite groups}

The arguments in this section are directly inspired by and reliant on Wilson's papers~\cites{wilson1,wilson2}.

The rough idea is the following.
If $G = \gen X$ is a residually finite group then $G$ is approximated by a sequence of finite quotients $G^* = G / N_i$, each of which has an induced generating set $X^*$, and \Cref{cor:theta-2-cor} applied to $G^* = \gen {X^*}$ gives
\[
    \diam(G^*) \le \eps (\log |G^*|)^{C \theta_2},
\]
where $C \ge 1$ is a constant, $\eps = \eps(G^*)$, and $\theta_2 = \theta_2(G^*)$.
This implies that there is an integer $m \le \eps (\log |G^*|)^{C\theta_2}$ such that
\[
    \gamma_X(m) \ge |G^*| \ge \exp((m / \eps)^{1/(C\theta_2)}).
\]
We would essentially be done at this point were it not for $\eps$, which is a nuisance.

We can avoid the dependence on the exponent parameter $\eps$ by arguing more circuitously.
In fact we will not apply our main diameter bounds directly but rather the results in \Cref{sec:anabelian-socle}.
We will argue that if $G$ has controlled growth then $G$ cannot have arbitrarily large nonabelian chief factors, and then deduce that $G$ is virtually residually soluble.
We will then use the following result of Wilson for residually soluble groups, which in turn depends on the deep result of Grigorchuk--Lubotzky--Mann~\cites{grig-residually-nilpotent,lubotzky--mann} on the residually nilpotent case (using group ring arguments and the Lazard analyticity criterion).

\begin{theorem}[Wilson~\cite{wilson2}*{Theorem~1}]
    \label[theorem]{thm:wilson}
    Let $G = \gen X$ be a residually soluble group such that
    \[\gamma_X(n) / \exp(n^{1/6}) \to 0\qquad \text{as}~n\to\infty.\]
    Then $G$ is virtually nilpotent.
\end{theorem}

\begin{remark}
    The exponent $1/6$ was recently improved to $1/4 - o(1)$ in \cite{EM25}.
\end{remark}

Concretely, the extension to residually finite groups is based on the following observation.

\begin{lemma}
    \label[lemma]{lem:growth-in-monolithic}
    Let $G = \gen X$ be a group with a finite quotient $G / K$ such that $G / K$ is monolithic with nonabelian socle $N = T^n$, where $T$ is simple.
    Let \[0 < \beta \le (19 + 10 \theta_2(G/K))^{-1}.\]
    Then there is a constant $c = c(\beta)$ and an integer $m \ge 0$ such that
    \[
        \gamma_X(m) \ge |N| \ge \exp(c m^\beta).
    \]
\end{lemma}
\begin{proof}
    There is no loss in assuming $K = 1$. Let $N = \soc(G) = T^n$.
    By \Cref{prop:monolithic-nonabelian-socle} and \Cref{lem:generalized-Milnor}, $N = \gen Z$ where
    \[
        \len_X(Z) \le 2n^{16 + 5 \theta_1(G)} \log_2 |T| + 2 n \log_2 |T|
        \ll n^{16 + 5 \theta_1(G)} \log |T|.
    \]
    By \Cref{thm:dona}, it follows that
    \begin{align*}
        \len_X(N) \le \len_X(Z) \diam(T^n)
        &\ll n^{19 + 5 \theta_1(G)} (\log |T|)^2 \diam(T) \\
        &\le n^{19 + 5 \theta_1(G)} (\log |T|)^{2 + \theta_2(G)} \\
        &\le (\log |N|)^{19 + 10 \theta_2(G)}\\
        &\le (\log |N|)^{1/\beta}.
    \end{align*}
    Let $C$ be a suitable constant such that $\len_X(N) \le C (\log |N|)^{1/\beta}$.
    Then the integer $m = \floor{C(\log |N|)^{1/\beta}}$ satisfies
    \[
        \gamma_X(m) \ge |N| \ge \exp(c m^\beta),
    \]
    where $c = C^{-\beta}$.
\end{proof}

\begin{lemma}
    \label[lemma]{lem:isolating-nonab-chief-factor}
    Suppose $G$ is a finite group with a nonabelian chief factor $N / K \cong T^n$. Then there is a monolithic quotient of $G$ with socle $T^n$.
\end{lemma}
\begin{proof}
    There is no loss in assuming $K = 1$.
    Then $N$ is a minimal normal subgroup of $G$.
    Now consider $G^* = G / C_G(N)$.
    Let $1 \ne L^* \nsgp G^*$ and let $L$ be the full preimage of $L^*$ in $G$.
    Since $L^* \ne 1$, $L$ is not contained in $C_G(N)$, so $1 \ne [L, N] \le L \cap N$.
    Since $N$ is minimal normal in $G$, $N \le L$, so $N^* \le L^*$.
    Thus $G^*$ is monolithic with socle $N^* \cong N$.
\end{proof}

Now we can prove that Babai's conjecture implies Grigorchuk's gap conjecture for residually finite groups.

\begin{proof}[Proof of \Cref{thm:gap-RF}]
    Given $\theta > 0$, let
    \[
        \beta_0 = 1/(19 + 10 \theta), \qquad \beta = 1 / (20 + 10 \theta).
    \]
    Now let $G = \gen X$ be a finitely generated group such that
    \begin{enumerate}
        \item $G$ has a sequence of finite-index normal subgroups $N_i$ such that $\bigcap_{i=1}^\infty N_i = 1$ and such that $\theta_2(G/N_i) \le \theta$ for all $i \ge 1$,
        \item $\gamma_X(n) \le C \exp(C n^\beta)$ for some constant $C > 0$ and all $n \ge 0$.
    \end{enumerate}

    Let $i \ge 1$ and suppose that $G / N_i$ has a nonabelian chief factor $T^n$.
    Then by \Cref{lem:growth-in-monolithic,lem:isolating-nonab-chief-factor} there is a constant $c_0 = c(\beta_0)$ and an integer $m \ge 0$ such that
    \[
        \gamma_X(m) \ge |T^n| \ge \exp(c_0 m^{\beta_0}).
    \]
    On the other hand $\gamma_X(m) \le C \exp(C m^\beta)$, so
    \[
        c_0 m^{\beta_0} \le C m^\beta + \log C,
    \]
    which implies that $m$ is $(\theta, C)$-bounded.
    Since
    \[
        |T^n| \le \gamma_X(m) \le C \exp(C m^\beta),
    \]
    it follows that $|T^n|$ is also $(\theta, C)$-bounded.
    Therefore $|\Aut(T^n)| \le a$ for some integer $a$ depending only on $(\theta, C)$.

    Let $H$ be the intersection of the subgroups of $G$ of index at most $a$. Note particularly that $a$ and $H$ are independent of $i$.

    Now again let $i \ge 1$. If $L / M \cong T^n$ is a nonabelian chief factor as above of $G^* = G / N_i$ then there is a natural map $G \to \Aut(T^n)$. Since $|\Aut(T^n)| \le a$, $H$ is contained in the kernel, so $H^*$ centralizes $T^n$.
    Let $G^* = G^*_0 > G^*_1 > \cdots > G^*_\ell = 1$ be a chief series of $G^*$. Then
    \[
        H^* = G^*_0 \cap H^* > G^*_1 \cap H^* > \cdots > G^*_\ell \cap H^* = 1
    \]
    is a normal abelian series for $H^*$.
    Indeed, each section $G^*_{i-1} \cap H^* / G^*_i \cap H^* \cong (G_{i-1}^* \cap H^*) G^*_i / G^*_i$ is $H^*$-isomorphic to a subgroup of $G^*_{i-1} / G^*_{i}$, which is either abelian or centralized by $H^*$.
    Thus $H^*$ is soluble.
    Since $i$ was arbitrary, it follows that $H$ is residually soluble.

    Finally, since $\beta < 1/6$, \Cref{thm:wilson} implies that $H$ is virtually nilpotent, and therefore so is $G$.
\end{proof}

\begin{proof}[Proof of \Cref{cor:branch-groups}]
    Let $d \ge 1$ and assume $G = \gen X \le \Aut(\T_d)$,
    where $\T_d$ is an infinite rooted $d$-regular tree.
    Let $\T_d^h$ denote the $d$-regular tree of height $h$.
    Let $N_h \nsgp G$ be the pointwise stabilizer of the vertices of $\T_d$ at distance $h$ from the root.
    Then
    \[
        G / N_h \le \Aut(\T_d^h) = \overbrace{S_d \wr \cdots \wr S_d}^h.
    \]
    If $T$ is a nonabelian composition factor of $G / N_h$ it follows that $T \hookrightarrow A_d$.
    Therefore $\theta_2(G / N_h)$ is bounded by some $\theta = \theta(d)$.
    Now let $\beta = \beta(\theta)$ be as in \Cref{thm:gap-RF}.
    Then $\gamma_X(n) \preceq \exp(n^\beta)$ implies that $G$ is virtually nilpotent.

    To give a more precise formula it suffices to bound
    \[
        \theta_2(T) = \frac{\log \diam(T)}{\log \log |T|}
    \]
    for nonabelian simple $T \le A_d$.
    Let $n = \mu(T)$. Then $5 \le n \le d$ and $T$ is a primitive simple subgroup of $S_n$.
    By \cite{maroti}*{Theorem~1.1}, either $|T| \ll n^{c \log n}$ or $T \cong A_m$ for some $m \le n$, in fact $m = n$ since $\mu(T) = n$.
    In the first case
    \[
        \theta_2(T) \le \log |T| \ll (\log n)^2.
    \]
    In the second case, by the Helfgott--Seress bound \eqref{eq:helfgott--seress},
    \[
        \theta_2(T) = \frac{\log \diam(A_n)}{\log \log (n!/2)} \ll (\log n)^3 \log \log n.
    \]
    Thus we may take $\theta(d) = c (\log d)^3 \log \log d$ for $d \ge 5$,
    where $c$ is some constant,
    and therefore we can take $\beta(d) = c' / ((\log d)^3 \log \log d)$ for some other constant $c' > 0$.
    Finally, $1 / (\log d)^4 < c' / ((\log d)^3 \log \log d)$ for $d$ sufficiently large.
\end{proof}

\begin{remark}
    Since growth functions are not totally ordered (trichotomy fails), there are multiple ways of formulating the gap conjecture.
    We are focusing on the weakest version.
    A stronger version of \Cref{conj:gap} (sometimes called $\Gap^*(\beta)$) predicts the existence of a constant $\beta > 0$ such that every non-virtually-nilpotent finitely generated group $G$ satisfies $\gamma_G(n) \succeq \exp(n^\beta)$ (rather than $\gamma_G(n) \not\prec \exp(n^\beta)$).
    An even stronger version is stated in \cite{EM25}*{Section~5}.
\end{remark}

\section{Examples}
\label[section]{sec:examples}

The form of the bound in \Cref{thm:main} suggests the intriguing possibility of precisely characterizing the structure of finite groups of polylogarithmic diameter, assuming the validity of Babai's conjecture (\Cref{conj:babai}).
For example, by \Cref{cor:anabelian}, anabelian is sufficient.
In general, to what extent is the dependence of \Cref{thm:main} on the exponent parameter
\[
\eps(G) = \max\{\exp(N / N') : N \nsgp G\}
\]
necessary? Can we say that only certain kinds of abelian sections are important?

In a similar spirit, Breuillard and Tointon~\cite{BT16}*{Section~4} showed that the \emph{polynomial} diameter condition $\diam(G) \ge |G|^\eps$ is equivalent to the presence of a cyclic section $H/N$ with index $|G:H| \ll_\eps 1$ and $|H : N| \gg_\eps |G|^\eps$.
If there is a similar condition that characterizes polylogarithmic diameter, it has to navigate the following examples.

Obviously cyclic groups have large diameter.
More generally, groups with large cyclic quotients have large diameter,
and dihedral groups are examples of groups with small abelianization but linear diameter due to a cyclic section.
More generally, if $H/N$ is a nontrivial $k$-generated abelian section of $G$ with $|G:H| \le k$ and $N \nsgp G$,
then $G/N$ is $(2k-1)$-generated by some set $X$,
there is a set $Y \subset H/N$ of Schreier generators, where $|Y| \le 2k^2$, and
\begin{equation}
    \label{eq:abelian-section-lower-bound}
    \diam(G) \ge \diam(G/N, X) \ge \diam(H/N, Y) \gg |H/N|^{1/(2k^2)}.
\end{equation}
In particular any large cyclic section $H/N$ with $N \nsgp G$ and $|G:H|$ small forces the diameter of $G$ to be large.

The presence of a large cyclic section $H/N$ may force the diameter of $G$ to be large even if the index $|G:H|$ is not small.
Consider the extension $G = V : A_n$, where $V \cong C_p^{n-1}$ is the standard deleted permutation module for $A_n$ over $\F_p$, $p \nmid n$.
Then \(\log |G| \sim n (\log p + \log n)\)
and Schreier's \Cref{lem:schreier} immediately gives $\diam(G) \ll pn \diam(A_n)$.
However, if $v = (1, -1, 0, \dots, 0) \in V$ then $X = \{v\} \cup A_n$ is a generating set for $G$ and it is easy to see that $\diam(G, X) \gg pn$.
Therefore
\[
    \max(pn, \diam(A_n)) \ll \diam(G) \ll p n \diam(A_n).
\]
It follows that $G$ has polylogarithmic worst-case diameter only if $p$ is very small, namely $p \le n^{O(1)}$ (and Babai's conjecture holds for $A_n$).
It is interesting that the best-case (or perhaps even typical-case) diameter remains small for much larger $p$:
the first and third authors~\cite{ES26} recently showed that as long as $p \le e^{O(n)}$ there is a bounded-cardinality generating set $X \subset G$ such that $\diam(G, X) \ll \log |G|$ (in fact, such that the Cayley graph is an expander graph).

On the other hand, not every large abelian section is harmful to the diameter.
A simple example is the wreath product $G = C_2 \wr C_n$, for which $\diam(G) \ll n^2$ by Schreier.
Another example is $G = \F_p^2 \rtimes \SL_2(p)$, for which $\diam(G) \ll (\log p)^{O(1)}$.
This follows from the following variant of \Cref{lem:schreier}\emph{(2)}.
Let $V \nsgp G$ be a normal abelian subgroup of a group $G$. The \emph{$G$-normal diameter} $\diam_G(V)$ of $V$ is defined to be the maximum value of $\diam(V, X)$ taken over $G$-invariant generating sets $X \subseteq V$. Then
\[
\diam(G) \ll \diam_G(V) \diam(G/V).
\]
Conversely, $\diam_G(V) \le \diam(G)$ if $V$ is complemented.

The previous examples may suggest that the critical thing is the maximal $G$-normal diameter of an abelian section, or possibly of a complemented abelian section, but neither is correct.
For example, let $p$ be an odd prime and consider an extension of the form $p^{1+2} : \SL_2(p)$ or $p^{1 + 4} : \SL_2(p)$ with centre $Z \cong C_p$.
In this case $Z$ is a large Frattini abelian subgroup of $G$, and $\diam_G(Z) = \diam(Z) \sim p/2$ since $Z$ is central, but one can check that $\diam(G)$ is polylogarithmic.
On the other hand, there are arbitrarily large Frattini abelian extensions of $A_5$ of the form $C_{2^n}^5 \cdot A_5$,
and $\diam(G) \ge |G|^c$ for any such group by \eqref{eq:abelian-section-lower-bound}
(the existence of these extensions follows from the theory of universal Frattini extensions: see Fried--Jarden~\cite{fried-jarden}*{Chapter~25} and \cite{fri95}*{\S{}II.C}).

Let us now turn away from diameters of abstract finite groups and focus on transitive permutation groups, which is our primary motivation.
It is not known whether any transitive subgroup of $S_n$ has greater than quadratic diameter.
In fact, apart from the giant cases $G = A_n$ and $G = S_n$ (and close relatives like $S_{n/2} \wr S_2$ etc),
it is generally difficult to construct transitive groups with diameter greater than linear.

One way of doing so is to consider the congruence quotients of the famous Grigorchuk group $G$. Let us briefly recall the construction. The interested reader might refer to \cite{dlH}*{Chapter~VIII} or \cite{mann}*{Chapter~10} or \cite{BGS-branch-groups-book}.
Let $\Omega = \{0,1\}^\omega$ be the set of all infinite binary strings.
We define involutions $a, b, c, d \in \Sym(\Omega)$ by the following recursive rules:
\begin{align*}
    &(0w)^a = 1w, && (1w)^a = 0w \\
    &(0w)^b = 0w^a, && (1w)^b = 1w^c, \\
    &(0w)^c = 0w^a, && (1w)^c = 1w^d, \\
    &(0w)^d = 0w, && (1w)^d = 1w^b.
\end{align*}
The \emph{first Grigorchuk group} is $G = \gen {a,b,c,d}$.
Note that $G$ may be identified with a subgroup of $\Aut(\T_2)$, where $\T_2$ is an infinite rooted binary tree,
so in particular $G$ is residually a finite $2$-group.
It turns out that $G$ is an infinite torsion group, and moreover $G$ has intermediate growth.
In fact it was proved by Bartholdi~\cite{bartholdi} that
\[
    \gamma_X(R) \preceq \exp(R^\beta),
\]
where $X = \{a, b, c, d\}$, $\beta = (1 - \log_2 \eta)^{-1} \approx 0.767$, and $\eta \approx 0.811$ is the real root of the polynomial $X^3 + X^2 + X - 2$.
Recently, a nearly matching lower bound was proved by Erschler and Zheng~\cite{erschler--zheng}:
\[
    \gamma_X(R) \succeq \exp(R^{\beta - o(1)}).
\]

By pruning the tree we obtain the \emph{congruence quotient} $G_h \le \Sym(\T_2^h)$, which is a $2$-group acting faithfully and transitively on $n = 2^h$ points.
Let $X_h$ be the image of $X$ in $G_h$.
By \cite{dlH}*{VIII.D.41},
\[
    |G_h| = 2^{5n/8 + 2} \qquad (h \ge 3).
\]
From Bartholdi's bound on the growth function it follows that
\[
    \diam(G_h, X_h) \gg n^{1/\beta}, \qquad 1 / \beta \approx 1.303.
\]
Unfortunately, we do not get a nearly matching upper bound on $\diam(G_h)$ from Erschler--Zheng, as it may be that a disproportionate number of the elements in the ball $B_X(R)$ are contained in the kernel of the map $G \to G_h$.
However, at least from \cite{bartholdi-lower}*{Proposition~7} we have
\[
    \diam(G_h, X_h) \ll n^{1.939}
\]
and, moreover, by \Cref{rem:bartholdi--bradford} the worst-case diameter is
\[
    \diam(G_h) \ll n^2.
\]

Variations on the Grigorchuk group have been the subject of intense study for several decades.
For a particularly vast generalization, we refer to the concept of a \emph{spinal group} as defined by Bartholdi--Grigorchuk--\v{S}uni\'{k}~\cite{BGS-branch-groups-book}.
The defining data consists of two groups, $A$ and $B$ (typically finite), an integer sequence $(d_i)_{i \ge 1}$ with $d_i \ge 2$ for each $i$, and a sequence of permutation representations
\[
    \alpha_1 : A \to \Sym(d_1),
    \qquad
    \alpha_{i,j} : B \to \Sym(d_i) \quad (i \ge 2, 1 \le j < d_{i-1}).
\]
Define
\[
    A_1 = \alpha_1(A),
    \qquad
    A_i = \gen{\alpha_{i,j}(B) : 1 \le j < d_{i-1}} \quad (i \ge 2).
\]
Let $\T$ be an infinite spherically homogeneous rooted tree with degree sequence $(d_i)_{i \ge 1}$.
We may identify the vertices of $\T$ with the set $\Omega$ of all finite strings $\omega = \omega_1 \cdots \omega_h$ with $1 \le \omega_i \le d_i$ for each $i$,
and we may identify $\Aut(\T)$ with the subgroup of $\Sym(\Omega)$ stabilizing the prefix relation.
We define actions $A, B \to \Aut(\T)$ as follows.
The action of $A$ on $\T$ is the obvious \emph{rooted action} defined by
\[
    (\omega_1 \cdots \omega_h)^a = \omega_1^{\alpha_1(a)}\omega_2 \cdots \omega_h \qquad (a \in A).
\]
The action of $B$ on $\T$ is the \emph{directed action} defined by the rule
\begin{align*}
    (d_1 \cdots d_{i-2} j \omega_i \cdots \omega_h)^b =
        d_1 \cdots d_{i-2} j \omega_i^{\alpha_{i,j}(b)} \omega_{i+1} \cdots \omega_h \\
        \hfill(b \in B, i \ge 2, 1 \le j < d_{i-1}),
\end{align*}
and $\omega^b = \omega$ if $\omega = d_1 \cdots d_{h-1} j$.
We assume both actions are faithful and we identify $A, B$ with their images in $\Aut(\T)$.
The group $G = \gen{A,B}$ is called a \emph{spinal group} if
\begin{enumerate}
    \item (spherical transitivity) $A_i \le \Sym(d_i)$ is transitive for each $i \ge 1$,
    \item (strong faithfulness) $\bigcap_{i \ge k} \bigcap_{j=1}^{d_{i-1} - 1} \ker \alpha_{i,j} = 1$ for all $k \ge 2$.
\end{enumerate}
%The ray $d_1d_2d_3\cdots$ is the \emph{spine} of the tree.
The subgroups $A$ and $B$ are the \emph{rooted part} and the \emph{directed part} of $G$.

Now assume $\alpha_{i,j}$ is trivial whenever $i \ge 2$ and $j \ne 1$, and let $\alpha_i = \alpha_{i,1}$ for $i \ge 2$.
We say $G$ is of \emph{type $\G$} if this holds and additionally
\begin{enumerate}[resume]
   \item (covering condition) $\bigcup_{i \ge j} \ker \alpha_i = B$ for all $j \ge 2$,
\end{enumerate}
and, more strongly, of \emph{type $\G_r$} if
\begin{enumerate}[resume]
    \item[($3_r$)] ($r$-uniform covering condition) $\bigcup_{i = j}^{j + r - 1} \ker \alpha_i = B$ for all $j \ge 2$.
\end{enumerate}
For example, the first Grigorchuk group is a spinal group of type $\G_3$ with
\[
    A = C_2, \qquad B = C_2 \times C_2,
\]
and a $3$-periodic sequence of homomorphisms $\alpha_i : B \to \Sym(2)$.

By \cite{BGS-branch-groups-book}*{Theorem~10.5--10.8}, every spinal group of type $\G$ with finite directed part has intermediate growth, and moreover the growth function satisfies
\[
    \gamma_G(n) \succeq \exp(n^{\beta_1}), \qquad \beta_1 = (1 + \log_m(2))^{-1}, ~ m = \min d_i.
\]
Furthermore, by \cite{BGS-branch-groups-book}*{Theorem~10.10}, if $G$ has type $\G_r$ then
\[
    \gamma_G(n) \preceq \exp(n^{\beta_2}), \qquad \beta_2 = (1 + \log_M(1/\eta_r))^{-1}, ~ M = \max d_i,
\]
where $\eta_r$ is the positive root of $X^r + X^{r-1} + X^{r-2} - 2$.

Consider an illustrative example. Take
\[
    A = C_2, \qquad B = C_2 \times C_2 \times \Alt(5).
\]
Let $(d_i)_{\ge 1}$ be a sequence in $\{2, 5\}$ such that $d_1 = 2$, $d_i = 2$ infinitely often, $d_i = 5$ infinitely often, and
\[
    d_i = 5 \implies d_{i+1} = d_{i+2} = d_{i+3} = 2 \qquad (i \ge 2).
\]
If $d_i = 2$ then there are three choices for
\[\alpha_i : B \to C_2 \times C_2 \to \Sym(2),\]
which we take in a $3$-periodic fashion.
If $d_i = 5$ then we take the natural map
\[\alpha_i : B \to \Alt(5) \to \Sym(5).\]
The resulting spinal group is of type $\G_4$, so the growth function of $G$ satisfies
\[
    \exp(n^{0.5}) \preceq \gamma_G(n) \preceq \exp(n^{0.922}).
\]
Now let $G_h$ be the quotient of $G$ by the level-$h$ stabilizer subgroup,
and let $X_h$ be the image of the generating set $X = A \cup B$ in $G_h$.
Then $G_h$ is a transitive subgroup of $\Sym(n)$ where $n = d_1 \cdots d_h$, $2^h \le n \le 5^h$,
every composition factor of $G_h$ is isomorphic either to $C_2$ or to $\Alt(5)$,
and the number of $\Alt(5)$ factors tends to infinity as $h \to \infty$.
One can show using spinal group arguments as in \cite{BGS-branch-groups-book} that the order of $G_h$ is exponential in $n$.
Finally, from the slow growth of $G$ it follows that
\[
    \diam(G_h, X_h) \gg n^{1.084}.
\]
In the other direction, at least it follows from \Cref{cor:transitive} and the fact that $\diam(A_5) = 10$ that
\[
    \diam(G_h) \ll n^{43}.
\]
% theta_1 ~= 1.43

Since $G$ is not virtually residually soluble, it is a counterexample to a conjecture of Sarah Black~\cite{black}*{Conjecture~$2'$} that any finitely generated residually finite group of subexponential growth is virtually residually soluble.
However, it is abundantly clear that the slow growth in $G$ is directly linked to the $C_2$ factors,
while the $A_5$ factors are only incidental.
The groups of Kassabov--Pak~\cite{kassabov-pak} are similar.
In fact we do not know examples of residually finite-anabelian groups of intermediate growth.
The following related question was recently posed by the first author in the Kourovka notebook~\cite{kourovka}*{21.44}.

\begin{question}
    Let $G$ be a finitely generated dense subgroup of the infinite iterated wreath product $W = \cdots \wr A_5 \wr A_5 \wr A_5$.
    Can $G$ have subexponential growth?
\end{question}

\Cref{cor:branch-groups} implies that there is a constant $\beta > 0$ such that no such group $G$ has growth $\gamma_X(n) \preceq \exp(n^\beta)$.
We may take $\beta = 1 / 40$.
% Why? diam(A_5) = 10, so theta_2 ~= 1.64 and we can take beta = 1/(20 + 10 \theta_2).

\bibliography{refs}
\end{document}